\documentclass[hyperref]{article}

\usepackage{tikz}
\usepackage{graphicx} 
\usepackage{float} 
\usepackage{amssymb}
\usepackage{amsmath}
\usepackage{mathtools}

\usepackage[thmmarks,amsmath]{ntheorem} 
\usepackage{bm} 
\usepackage[colorlinks,linkcolor=cyan,citecolor=cyan]{hyperref}
\usepackage{extarrows}
\usepackage[hang,flushmargin]{footmisc} 
\usepackage[square,comma,sort&compress,numbers]{natbib} 
\usepackage{mathrsfs} 
\usepackage[font=footnotesize,skip=0pt,textfont=rm,labelfont=rm]{caption,subcaption} 
\usepackage{booktabs} 
\usepackage{tocloft}
\usepackage{stmaryrd}
{
    \theoremstyle{nonumberplain}
    \theoremheaderfont{\bfseries}
    \theorembodyfont{\normalfont}
    \theoremsymbol{\mbox{$\Box$}}
    \newtheorem{proof}{Proof}
}

\usepackage{theorem}
\newtheorem{theorem}{Theorem}[section]
\newtheorem{lemma}{Lemma}[section]

{
    \theoremheaderfont{\bfseries}
    \theorembodyfont{\normalfont}
    
}
\graphicspath{{figure/}}                                 
\usepackage[a4paper]{geometry}                           
\begin{document}
\newcommand{\topcaption}{%
\setlength{\abovecaptionskip}{0.cm}%
\setlength{\belowcaptionskip}{0.cm}%
\caption}

\title{\bf Local multiscale model reduction using discontinuous Galerkin coupling for elasticity problems } 
\date{}
\author{\sffamily Zhongqian Wang$^1$, Shubin Fu$^{2,*}$, Eric Chung$^1$\\
    {\sffamily\small $^1$ Department of Mathematics, The Chinese University of Hong Kong, Hong Kong SAR }\\
    {\sffamily\small $^2$ Department of Mathematics, University of Wisconsin - Madison, USA }}
\renewcommand{\thefootnote}{\fnsymbol{footnote}}
\footnotetext[1]{Corresponding author. }
\maketitle

{\noindent\small{\bf Abstract:}
 	 In this paper, we consider the  constraint energy minimizing generalized multiscale finite element method (CEM-GMsFEM) with discontinuous Galerkin (DG) coupling for the linear elasticity equations in highly heterogeneous and high contrast media. We will introduce the construction of a DG version of the CEM-GMsFEM, such as auxiliary basis functions and offline basis functions. 
	 The DG version of the method offers some advantages such as flexibility in coarse grid construction and sparsity of resulting discrete systems. 
	 Moreover, to our best knowledge, this is the first time where the proof of the convergence of the CEM-GMsFEM in the DG form is given. 
	 Some numerical examples will be presented to illustrate the performance of the method. 
	 }

\vspace{1ex}
{\noindent\small{\bf Keywords:}
    Multiscale finite element method; CEM-GMsFEM; discontinuous Galerkin; elasticity problem}
\section{Introduction}
Many problems in science and engineering exhibit high heterogeneity.  One example is the seismic meta material for which the Young's modulus of Foamed plate is about $1.6\times 10^5$Pa, and it  is about $2.07\times 10^{11}$Pa for steel \cite{zeng2019broadband}. In order to solve problems in highly heterogeneous media, one usually needs to  discretize the governing equation with  high resolution mesh in order to obtain reliable solutions, and this will inevitably results in large-scale systems which are hard and sometime even impossible to solve with current computing resources. Therefore, some model reduction techniques  are necessary.  There are some existing model reduction techniques for handling highly heterogeneous media or material. One way is to homogenize the detailed heterogeneous media based on some rules so that one can  solve the equation with low resolution grid \cite{oleinik2009mathematical,DDDAS_upscale_2004,Arbogast_Boyd_06,chung2001asymptotic}. Another way is to solve the equation  in a coarse grid using multiscale basis functions. These functions are usually constructed by solving some carefully designed local problems, and they include heterogeneous information of the media \cite{egw10,efendiev2009multiscale,Arbogast_two_scale_04}.
 
There are in literature various multiscale model reduction techniques based on different types of multiscale methods for highly heterogeneous problems. Some examples include the multiscale finite element method \cite{hou1997multiscale,chen2003mixed}, variational multiscale method \cite{hughes1998variational,hughes1998variational}, multiscale mortar methods \cite{arbogast2007multiscale,wheeler2012multiscale}, localized orthogonal decomposition method\cite{maalqvist2014localization}, multiscale finite volume method \cite{cortinovis2014iterative,lunati2004multi,jenny2003multi}. Among these methods, the multiscale finite element method (MsFEM) and its generalization the Generalized Multiscale Finite element method (GMsFEM) have achieved success for various types of heterogeneous problems \cite{efendiev2009multiscale, vasilyeva2019constrained, wang2021comparison}. The basic idea of both the MsFEM and the GMsFEM is to solve the equation in a relative low resolution grid with  appropriately constructed multiscale basis functions  instead of polynomial basis functions as in traditional finite element methods. These basis functions are constrained forms trusted by solving some  local heterogeneous problems with specific boundary conditions and they include small scale medium information. Thus these basis functions allow us to perform  accurate but  inexpensive coarse-grid simulations.
Motivated by the localized orthogonal decomposition method (LOD) \cite{maalqvist2014localization,henning2014localized,engwer2019efficient,maalqvist2020numerical} 
the constraint energy minimizing generalized multiscale finite element method (CEM-GMsFEM) \cite{chung2018constraint2} is developed. 
The CEM-GMsFEM basis functions are constructed based on local spectral problems and an energy minimization principle. The resulting basis functions have exponential decay
property for high contrast media. Moreover, the convergence depends only on the coarse mesh size and is independent of the scale and contrast of the media.
The method has shown success  for many problems, for example \cite{cheung2020iterative,vasilyeva2019constrained}. The original CEM-GMsFEM is based on continuous Galerkin coupling. It is extended to the discontinuous Galerkin (DG) formulation in \cite{cheung2021explicit}. The DG coupling allows the basis functions to be discontinuous and thus gives more flexibility for complex applications. 

In this paper, we consider the CEM-GMsFEM  with discontinuous Galerkin (DG) coupling for  the linear elasticity equations in highly heterogeneous media. We use the classical  internal penalty discontinuous Galerkin (IPDG) method as the coarse grid scheme to couple the multiscale basis functions. The construction of multiscale basis functions consists of two major steps. In particular, we will construct auxiliary basis functions by solving local spectral problems on each coarse element, which is followed by  solving  energy minimization problems in oversampled domains. This will  yield the final multiscale basis functions. We propose two variants of energy minimization problems. The first one is based on solving constraint energy minimization problems while the second one is the relax version by solving unconstrained energy minimization problems \cite{fu2020elastic,fu2020constraint}. In comparison with previous convergence methods, the form of the CEM-GMsFEM under DG coupling proposed in this paper provides better convergence for elastic problems. By choosing an appropriate number of oversampling and basis functions, an accurate approximation can be achieved even with very large contrasts. 
Rigorous analysis and numerical examples for both versions will be provided.

This paper will be organized as follows. In Section \ref{section2}, we present the isotropic elasticity problem in heterogeneous media and the IPDG method. 
In Section \ref{section3}, CEM-GMsFEM is presented. In Section \ref{section4}, we provide some stability and convergence results. In Section \ref{section5}, we give some numerical results to illustrate the performance of this method. Finally, we conclude the paper in section \ref{section6}.

\section{Preliminaries}\label{section2}
In this section, we present the isotropic elasticity problem in heterogeneous media and the DG method. 
We will also introduce the mesh required for our method, and some norms for our analysis. 

\subsection{Isotropic elasticity problem in heterogeneous media }
In this paper, we consider the following elastic problem in the domain $\Omega \subset \mathbb{R}^2$,
\begin{subequations}
\begin{align}
\mathrm{div}\left ( \sigma\left( u\right)\right )+ f&=0,\label{1a} \\
\sigma &=C: \varepsilon,\label{1b} \\
\varepsilon &=\frac{1}{2}\left[\mathrm{grad}\ u+(\mathrm{grad}\ u)^{\mathrm{T}}\right],\label{1c}
\end{align}
\end{subequations}
 where $u$ is the displacement field, and 
 $f $ is the given source term. 
In the above system, $\sigma = \sigma \left ( u \right )$ represents the stress tensor, $ \epsilon =\epsilon \left ( u \right )$ represents the strain tensor, $C=C_{i,j,p,q}\left ( x \right ) $ is the fourth-order tensor
representing material properties where $i,j,p,q=1$  or $3$ for two-dimensional space. 
Using Voigt notation, the elastic tensor can be expressed in terms of the following coefficient matrix 
\begin{equation}
     C=\begin{bmatrix}
C_{11} & C_{13} & 0\\ 
 C_{31}&  C_{33}& 0\\ 
 0& 0 & C_{55}
\end{bmatrix}.
\end{equation}
 In addition, for isotropic elastomer, its constitutive equation is {$ \sigma (u)=\lambda tr\left ( \epsilon (u) \right )I+2\mu \epsilon (u) $}, and $\lambda$ and $\mu$ are Lame constants satisfying $\lambda +\mu >0,\mu >0$. 
So, we have $C_{11}=\lambda +2\mu ,$  $ C_{33}=\lambda +2\mu , $  $C_{13}=\lambda ,$  $ C_{55}=\mu .$
We recall that,
in (\ref{1c}), $\mathrm{grad}\, u=\left ( \frac{\partial u_{i}}{\partial x_{j}} \right )_{1\leq i,j\leq 2}$, and 
\begin{equation}
    \epsilon _{ij}\left ( u \right )=\frac{1}{2}\left ( \frac{\partial u_{i}}{\partial x_{j}}+\frac{\partial u_{j}}{\partial x_{i}}\right), 1\leq i,j\leq 2.
\end{equation}

\subsection{Interior Penalty Discontinuous Galerkin (IPDG) method}

We first introduce some concepts of grids as illustrated in \autoref{grid}. 
Let $\mathcal{T}^{H}$ be a coarse grid for the domain $\Omega$. An element $K$ in $\mathcal{T}^{H}$  is called a coarse element, where $K$ is a triangle or quadrilateral in the two-dimensional case.
We call $H$ the size of the coarse grid, and we denote the diameter of the cell $K$ as $h_{K}$. 
We let $N_{c}$ be the total number of coarse nodes, and $N$ be the number of coarse elements. We also denote the collection of all coarse grid edges by $\mathcal{E}^{H} $.
For any real number $s$, the discrete Sobolev space $H^{s}\left ( \mathcal{T}^{H} \right )$ is defined as follows:
$$H^{s}\left ( \mathcal{T} ^{H} \right )=\left \{ v\in L^{2} \left ( \Omega  \right ):\forall K\in \mathcal{T}^{H},v\mid_{K}\in H^{s}\left ( K \right ) \right \}.$$
The corresponding discrete Sobolev norm is defined as follows:
\begin{equation}
\left \| v \right \|_{s,\mathcal{T} ^{H}}=\left (\sum_{K\in \mathcal{T}^{H}}\left \| v \right \|_{s,K}^{2}  \right)^{1/2},
\end{equation}
where $\|\cdot\|_{s,K}$ is the norm on $H^s(K)$.
Besides, we denote the seminorm of the discrete gradient as
\begin{equation}
\left \| \mathrm{grad}\ v \right \|_{s,\mathcal{T} ^{H}}=\left (\sum_{K\in \mathcal{T} ^{H}}\left \| \mathrm{grad}\  v \right \|_{s,K}^{2}  \right)^{1/2}.
\end{equation}
In addition, we define the fine grid as $\mathcal{T}^{h}$, and $h> 0$ is the size of the fine grid. 
We note that the fine grid is a refinement of the coarse grid $\mathcal{T}^H$.
In \autoref{grid}, we illustrate an example of coarse and fine grids when rectangular partitions are adopted.
In the rest of the paper, we assume this type of rectangular partition is used. 

\begin{figure}[H] 
\centering 
\includegraphics[width=0.7\textwidth]{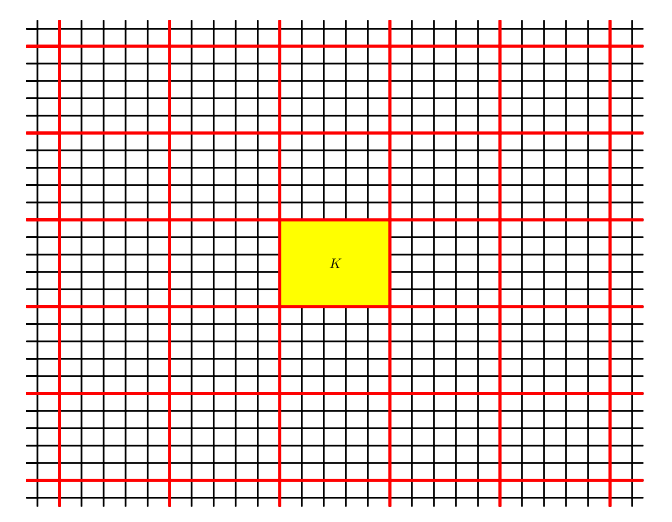} 
\caption{Illustration of the grids. We use black to denote the fine grid and red to denote the coarse grid. A coarse element $K$ is shown.}
\label{grid} 
\end{figure}

Next we give the formulation of the DG method. For each coarse block $K_{i}$, we denote the restriction of the Sobolev space $H^{1}(\Omega)$ on $K_{i}$ by $V_h\left(K_{i}\right) $. That is, for the discrete space,  $V_h\left(K_i\right)$ is the space of piecewise bilinear functions with respect to the fine grid in $K_i$. 
Then the global DG space is 
$$V_{h}=\bigoplus_{i=1}^{N} V_{h}\left(K_{i}\right).$$
We remark that functions in $V_{h}$ are continuous within coarse blocks, but discontinuous across the coarse grid edges in general. The formulation of IPDG method then reads: find $u_{h} \in V_{h}$ such that
\begin{equation}
    a_{\text{DG}}\left ( u_{h}, v \right )=\int_{\Omega }  f\cdot  v \; \text{dx}, \quad \forall \  v \in V_{h}, \label{adg}
\end{equation}
where the bilinear form $a_{\text{DG}}\left ( u,v \right )$ is defined as 

\begin{equation}
\begin{aligned}
a_{\text{DG}}\left (  u, v \right )&=\sum_{K\in \mathcal{T} ^{H}}\int_{K} \sigma \left (  u \right ):\epsilon \left ( v \right )\text{dx}\\
&-\sum_{E\in \mathcal{E}^{H}}\int_{E}\left ( \left \{ \sigma\left( u\right) \right \} :\underline{{{[\![ v]\!]}} }+\eta \underline{{{[\![ u]\!]}}}:  \left \{  \sigma\left( v\right) \right \} \right)\text{ds}\\
&+\sum_{E\in \mathcal{E} ^{H}}\frac{\gamma }{h}\int_{E}\left (\underline{{{[\![ u]\!]}}}:\left \{  C \right \}:\underline{{{[\![ v]\!]}}} +{[\![ u]\!]} \cdot \left \{ D\right \} \cdot {[\![ v]\!]}\right )\text{ds},
\end{aligned}
\end{equation}
where \begin{equation}
    D=\begin{bmatrix}
C_{11} & 0 & 0\\ 
 0& C_{33} &0 \\ 
 0& 0 & C_{55}
\end{bmatrix}
\end{equation}
and $\eta$ can take the value of $-1, 0$ or $1$, where we pick $\eta=1$ in this paper that corresponds to the classical Symmetric Interior Penalty Galerkin (SIPG) method. Here $\gamma>0$ is the penalty parameter. Besides, $\eta=0$  is for the Incomplete Interior Penalty Galerkin (IIPG) method and $\eta=-1$  is for the Non-symmetric Interior Penalty Galerkin (NIPG) method. In the above bilinear form, $[\![ v]\!]$ is the vector jump and \underline{{{[\![v]\!]}}} is the matrix jump to be defined as follows. For coarse grid edges in the interior of the domain, we define 
\begin{equation}
    {{[\![ v]\!]}}=v^{+}-v^{-},
\end{equation}
\begin{equation}
\underline{{{[\![v]\!]}}}=v^{+} \otimes n^{+}+v^{-} \otimes n^{-},
\end{equation}
where $'+'$ and $'-'$ respectively represent the values on two adjacent cells $K^+$ and $K^-$ sharing the coarse grid edge, and $n^+$ and $n^-$
are the unit outward normal vector on the boundary of $K^+$ and $K^-$. In addition, 
we define the average of a tensor $\sigma $ as 
\begin{equation}
    \left \{  \sigma \right \} =\frac{\sigma ^{+}+ \sigma ^{-}}{2},
\end{equation}
where $\sigma ^{+},$  $  \sigma ^{-}$ are two tensors defined on $K^+$ and $K^-$ sharing a common interior coarse edge.
For the coarse grid edges lying on the boundary of the domain,  
we define
\begin{equation}
    \left \{  \sigma \right \} =\sigma, \  [\![v]\!]= v, \     \underline{\llbracket v \rrbracket}=v \otimes n,
\end{equation}
where $n$ is the normal vector pointing outside of the domain.
For each coarse element $K_i$, we define the oversampling domain $K_{i,p}$ by extending $K_i$ by $p$ coarse grid layers. 
In addition, for the $i$-th coarse grid node, we let $\omega_i$ be the coarse neighborhood defined by the union of all coarse elements having the $i$-th node.
\autoref{grid2} shows an example of an oversampling domain for $p=1$ and a coarse neighborhood.

\begin{figure}[H] 
\centering 
 \includegraphics[width=0.9\textwidth]{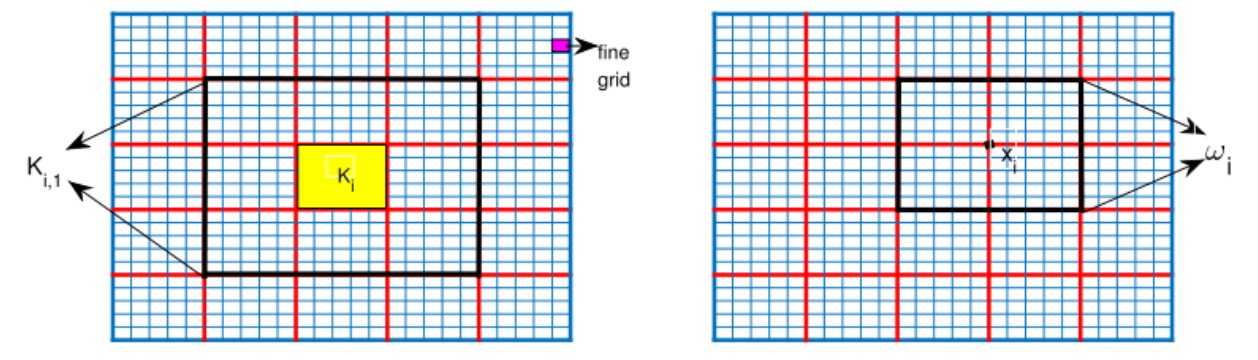} 
 \caption{Oversampling grid illustration.}
 \label{grid2} 
 \end{figure}

\section{Construction of multiscale basis functions}\label{section3}

In this section, we will present the construction of the multiscale space $V_{\text{cem}}$ that will be used in our multiscale method.
The main steps include generating multiscale auxiliary functions and then computing multiscale basis functions for global coupling. 
\subsection{Multiscale auxiliary functions}
We first show how to construct auxiliary functions. For the $j\text{-}\rm{th}$ coarse block $K_{j}$,  recall that $V_{h}\left(K_{j}\right)$ is the space of piecewise bilinear elements defined on $\mathcal{T} ^{h}$ in $K_{j}$. 
We compute $\left ( \lambda _{i}^{j},\psi _{i}^{j} \right )\in \mathbb{R}\times V_{h}\left(K_{j}\right)$ by solving the following eigenvalue problem
\begin{equation}
    a_{j}\left(\psi _{i}^{j},v\right)=\lambda _{i}^{j}b_{j}\left(\psi _{i}^{j},v\right),\ \forall \ v\in V_{h}\left ( K_{j} \right )\label{abf},
\end{equation}
in which
\begin{equation}
    a_{j}\left ( w,v \right )=\int_{K_{j}} \sigma \left(w\right) :  \epsilon \left(v\right)\text{dx},
\end{equation}
\begin{equation}
     b_{j}(w, v)=\int_{K_{j}}k_1 wv \text{dx},  \label{eq:bj}
\end{equation}
$\{ \chi_{i} \}_{i=1}^{N_c}$ is a set of partition of unity functions for the coarse grid, $k_2=\lambda + 2\mu $  and $k_1=\sum_{i=1}^{N_{c}}k_2\left|\nabla \chi_{i}\right |^{2}.$
We remark that the choice of the bilinear form (\ref{eq:bj}) is motivated by our convergence analysis.
The term $\left|\nabla \chi_{i}\right |^{2}$ is merely a scaling factor, and is not crucial in our analysis. The most important part in (\ref{eq:bj})
is the inclusion of $k_2$. Such choice gives small and contrast dependent eigenvalue for each high contrast channel network in a coarse element.
See \cite{fu2020constraint,fu2020elastic} for more details.
We will arrange the eigenvalues of the above problem (\ref{abf}) in non-decreasing order.
The local auxiliary space $W_{\text{aux}}\left ( K_{j} \right )$ is formed by the first $G_{j}$ eigenfunctions corresponding to the first $G_{j}$ smallest eigenvalues, that is, 
\begin{equation}
    W_{\text{aux}}\left ( K_{j} \right )=\text{span}\left \{ \psi _{i}^{j}\ |1\ \leq i\leq G_{j} \right \}.
\end{equation}
The global multiscale auxiliary space $W_{\text{aux}}$ is the sum of the local auxiliary spaces as follows,
\begin{equation}
    W_{\text{aux}}=\bigoplus_{j=1}^{N}W_{\text{aux}}\left(K_{j}\right).
\end{equation}
We next define the notion of $\psi_{i}^{j}$-orthogonality in the space $V_{h}$. 
We let $b\left ( u,v \right )=\sum_{j=1}^{N}b_{j}\left (u,v  \right )$.
For each $\psi_{i}^{j} \in W_{\text{aux}}$, we say that $\zeta\in V_h $ is $\psi_{i}^{j}$-orthogonal if
\begin{equation}
    b\left ( \zeta ,\psi _{i'}^{j'} \right )=\left\{\begin{matrix}
0,\ &\text{if}\quad {j}'\ \neq \ j, \text{ or } \ {i}'\neq i, \\ 
1,\ &\text{otherwise}.
\end{matrix}\right.
\end{equation}
Finally,
a projection operator $\pi:V_{h}\rightarrow W_{\text{aux}}$ is defined by
\begin{equation}
     \pi\left ( v \right )=\sum_{j=1}^{N}\sum_{i=1}^{G_{j}}b_{j}\left ( v,\psi _{i}^{j} \right )\psi _{i}^{j}, \ \forall \ v\in V_{h}.
 \end{equation}

\subsection{Multiscale basis functions}
In this section, we construct the multiscale trial basis functions. 
Let $\psi_{i}^{j}\in W_{\text{aux}}$ be the $i$-th eigenfunction of the spectral problem (\ref{abf}) for the coarse element $K_j$.
We will construct a global basis function $\phi_{i,\text{glo}}^{j}$ and a local basis function $  \phi_{i,\text{ms}}^{j}$.
The global basis function has support in the whole computational domain. The global basis functions span an approximation space that gives good approximation property.
We will show that the global basis functions have an exponentially decay property, which states that the global basis functions
have small values outside an oversampling region. 
The properties suggest the construction of local basis functions, which can be computed efficiently by solving local problems and are more suitable for numerical computations.

We  first consider the global version of the basis functions. 
For each multiscale auxiliary function $\psi_{i}^{j}\in W_{\text{aux}}$, there is a  multiscale trial basis function $\phi_{i,\text{glo}}^{j}$ that satisfies the following constraint energy minimization problem 
\begin{equation}
    \phi_{i,\text{glo}}^{j}={\rm argmin}\left \{ a_{\text{DG}}\left ( \phi ,\phi  \right):\pi\left(\phi\right)=\psi_{i}^{j}\right \}\label{basisiob},
\end{equation}
where the minimum is sought in the space $V_h$. 
We note that  $\phi_{i,\text{glo}}^{j}$ is $\psi_{i}^{j}$-orthogonal.
By using the Lagrange multiplier, the above minimization problem can be rewritten as : find $\phi _{i,\text{glo}}^{j}\in V_h$, 
$\delta _{i}^{j}\in W_{\text{aux}}$ such that
\begin{equation}
\begin{aligned}
    a_{\text{DG}}\left (\phi_{i,\text{glo}}^{j}, \phi\right )+b\left ( \phi,\delta _{i}^{j} \right )&=0,\ \forall\  \phi \in V_h,\\b\left ( \phi _{i,\text{glo}}^{j}-\psi_{i}^{j},\delta  \right )&=0,\  \forall\  \delta  \in W_{\text{aux}}.
\end{aligned}
\end{equation}
Then we define the global multiscale space
\begin{equation}
    V_{\text{cem}}=\text{span}\left\{\phi _{i,\text{glo}}^{j}:1\leq i\leq G_{i},1\leq j\leq N\right\}.
\end{equation}

Now, we define our local multiscale trial basis functions. This construction is motivated by an exponential decay property
of the basis function $\phi^j_{i,\text{glo}}$. Therefore,
local multiscale basis functions are constructed on some oversampling regions. We introduce the subspace $V_{h}\left(K_{j,p}\right)$
defined by the union of all $V_h(K_m)$ for $K_m \subset K_{j,p}$.
Similarly, the subspace $W_{\text{{aux}}}\left(K_{j,p}\right)$ is defined by the union of all $W_{\text{aux}}\left(K_m\right)$ for $K_m \subset K_{j,p}$.
Given a function $\psi^j_i \in W_{\text{aux}}(K_{j,p})$, 
the local multiscale trial basis function $\phi_{i,\text{ms}}^{j}\in V_{h}(K_{j,p})$ is defined as 
\begin{equation}
    \phi_{i,\text{ms}}^{j}=\underset{\phi\in V_{h}\left(K_{j,p}\right)}{\text{argmin}}\left \{ a_{\text{DG}}(\phi,\phi):\pi\left ( \phi \right )=\psi _{i}^{j} \right \},\label{minimization}
\end{equation}
where the zero Dirichlet boundary condition on $\partial K_{j,p}$ is weakly imposed.
Similarly, we use the Lagrange multiplier to write the above problem as: find\ $\phi_{i,\text{ms}}^{j}\in V_{h}\left(K_{j,p}\right)$ and $\delta_{i,\text{ms}}^{j}\in W_{\text{aux}}\left(K_{j,p}\right)$ such that
\begin{equation}
\begin{aligned}
    a_{\text{DG}}\left (\phi_{i,\text{ms}}^{j}, \phi\right )+b\left ( \phi,\delta _{i,p}^{j} \right )&=0,\  \forall\ \phi \in V_{h}\left(K_{j,p}\right),\\b\left ( \phi_{i,\text{ms}}^{j}-\psi_{i,p}^{j},\delta  \right )&=0,\ \forall\ \delta \in W_{\text{aux}}\left(K_{j,p}\right).
\end{aligned}
\end{equation}
Then we get the localized multiscale trial space, i.e.
\begin{equation}
    V_{H}=\text{span}\left\{\phi _{i,\text{ms}}^{j}:1\leq i \leq G_{j},1\leq j\leq N \right\}.
\end{equation}


\subsection{Relaxed minimization}
\label{sec:relax}

We recall that the basis functions in the previous section solve constraint energy minimization problems (\ref{minimization}).
These constraints can be relaxed. In particular, we consider the following unconstrained energy minimization problems for the constructions of basis functions. 
Given a function $\psi^j_i \in W_{\text{aux}}(K_{j,p})$, we compute $\phi_{i,\text{ms}}^j \in V_h(K_{j,p})$ by
\begin{equation}
    \phi _{i,\text{ms}}^{j}=\text{argmin}\left \{ a_{\text{DG}}\left ( \phi ,\phi  \right ) +b\left ( \pi\left ( \phi  \right )-\psi _{i}^{j} ,\left ( \phi  \right )-\psi _{i}^{j} \right )|\ \phi \in V_{h}\left ( K_{j,p} \right )\right \}\label{basisrelexip}.
\end{equation}
This problem is equivalent to the following variational formulation
\begin{equation}
    a_{\text{DG}}\left ( \phi _{i,\text{ms}}^{j},w \right )+b\left ( \pi\left ( \phi _{i,\text{ms}}^{j} \right ) ,\pi\left ( w \right )\right )=b\left ( \phi _{i,\text{ms}}^{j}, \pi\left ( w \right )\right ),\quad  \forall \ w\in V_{h}\left ( K_{j,p} \right )\label{problemrelaxip}.
\end{equation}
In \autoref{example2msbasisrelaxed} and  \autoref{example2msbasisconstraint}, we show examples of multiscale basis functions, from which we can see clearly their 	decay behaviors.

Furthermore, 
the global multiscale basis function $\phi _{i,\text{glo}}^{j}\in V_h$ is defined as 
\begin{equation}
    \phi _{i,\text{glo}}^{j}=\text{argmin}\left \{ a_{\text{DG}}\left ( \phi ,\phi  \right ) +b\left ( \pi\left ( \phi  \right )-\psi _{i}^{j} ,\left ( \phi  \right )-\psi _{i}^{j} \right )|\ \phi \in V_h\right \}\label{basisrelexiob}.
\end{equation}
This problem is equivalent to the following variational formulation
\begin{equation}
    a_{\text{DG}}\left ( \phi _{i,\text{glo}}^{j},v \right )+b\left ( \pi\left ( \phi _{i,\text{glo}}^{j} \right ) ,\pi\left ( w \right )\right )=b\left ( \phi _{i,\text{glo}}^{j}, \pi\left ( w \right )\right ),\quad  \forall \ w\in V_h\label{problemrelaxiob}.
\end{equation}

Finally, the multiscale solution $u_{\text{ms}}^{\text{off}}$ is defined as the solution of the following problem, find $u_{\text{ms}}^{\text{off}}\in V_{H}$, such that 
\begin{equation}
\label{eq:ms}
    a_{\text{DG}}\left (u_{\text{ms}}^{\text{off}},v \right )=\int_{\Omega } fv\ \text{dx},  \ \forall \ v \in V_{H}.
\end{equation}

\begin{figure}
  \centering
  \includegraphics[scale=0.4]{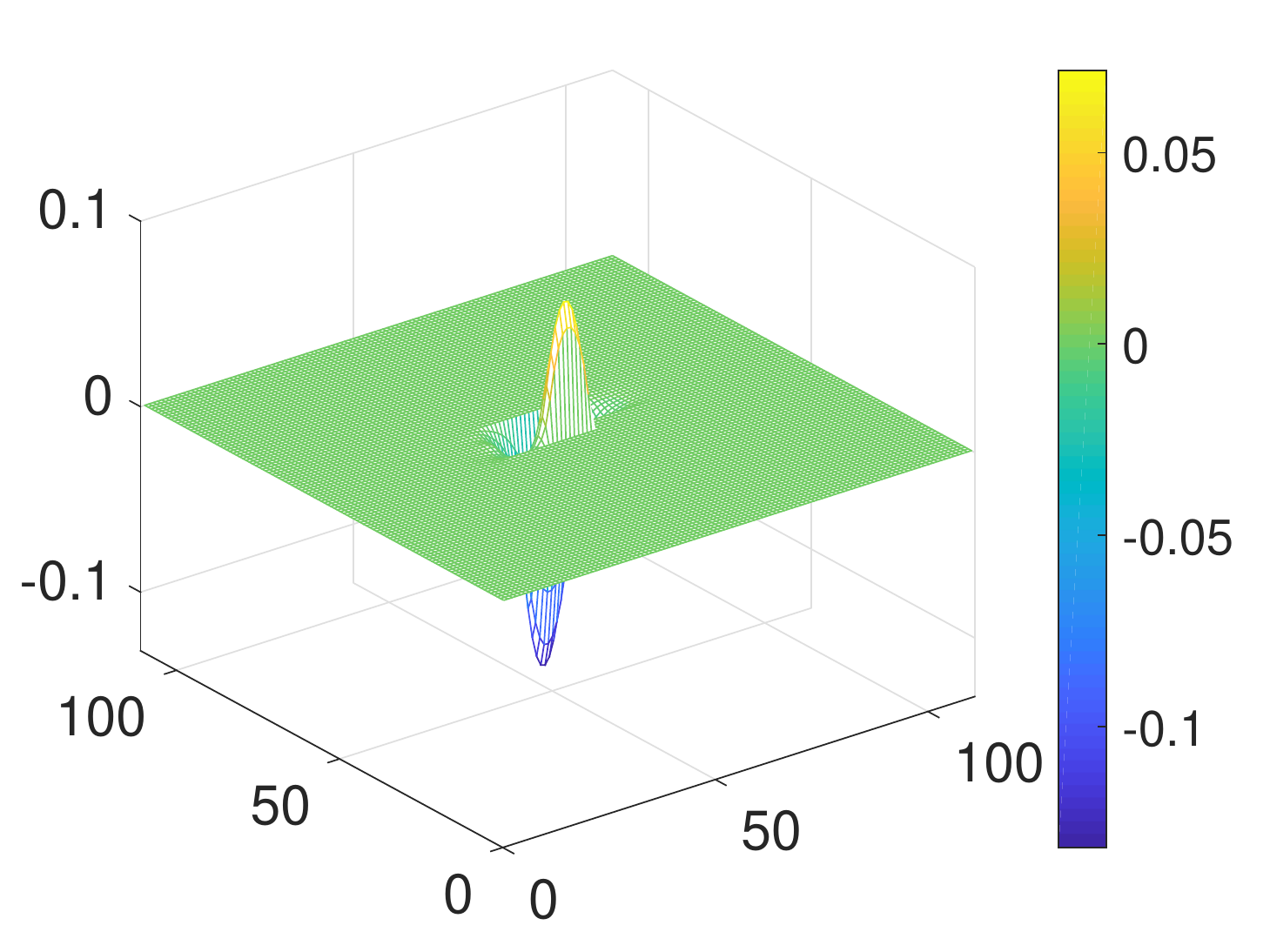}
  \hspace{0.5in}
  \includegraphics[scale=0.4]{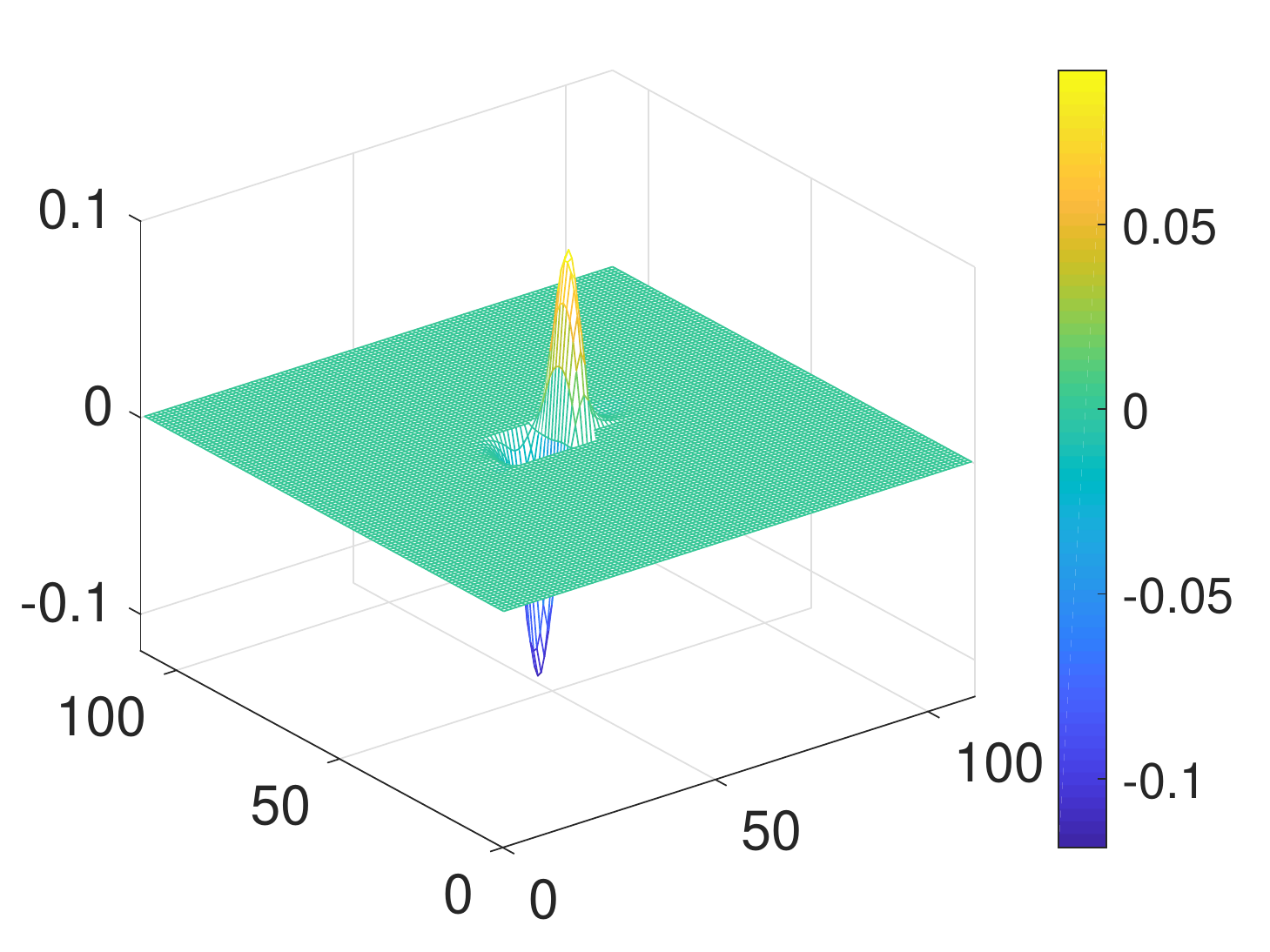}
  \caption{An example of  relaxed CEM-GMsFEM basis function. First component (left),  and second component (right).}
  \label{example2msbasisrelaxed} 
\end{figure}
\begin{figure}
  \centering
  \includegraphics[scale=0.4]{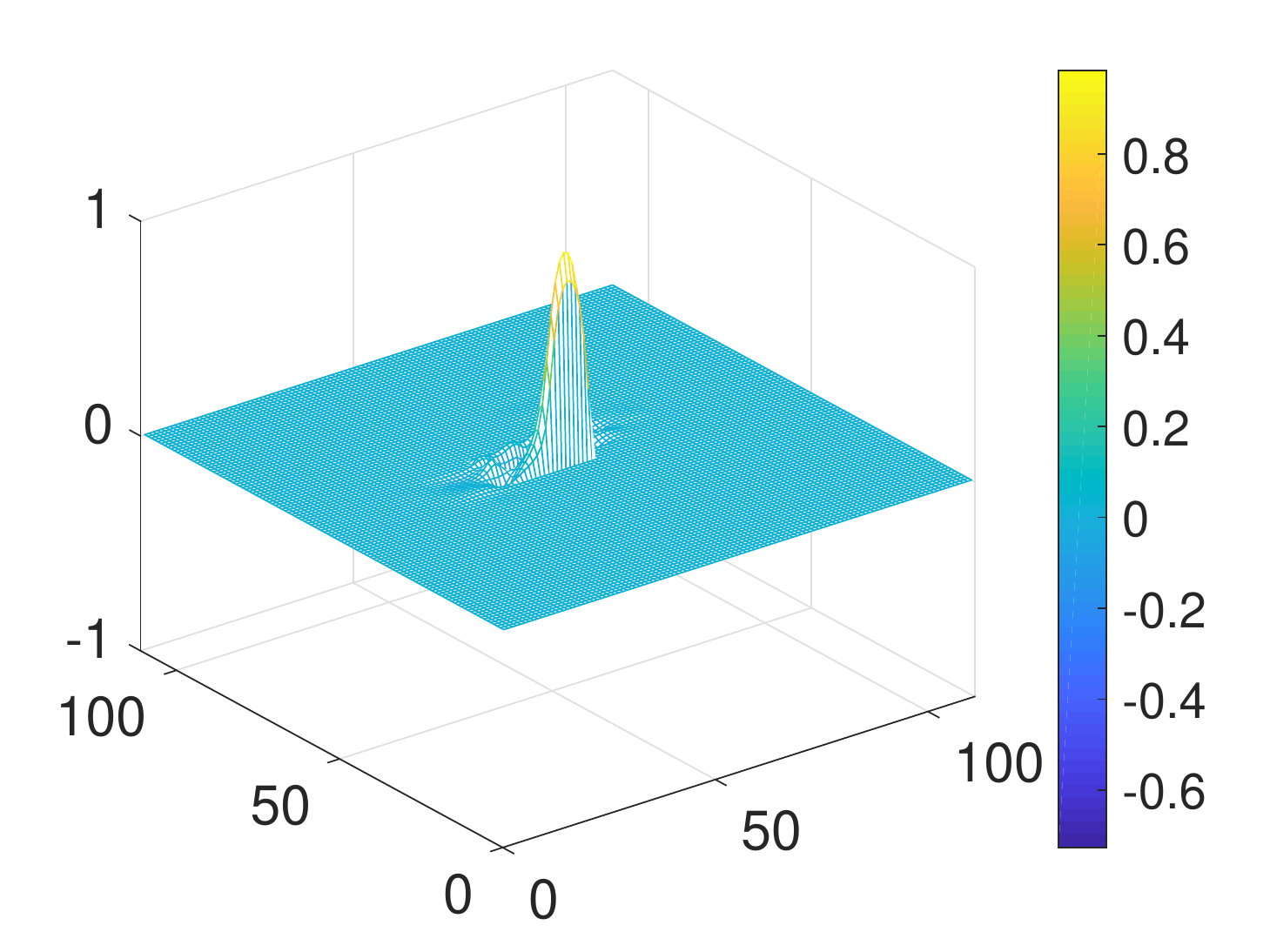}
  \hspace{0.5in}
  \includegraphics[scale=0.4]{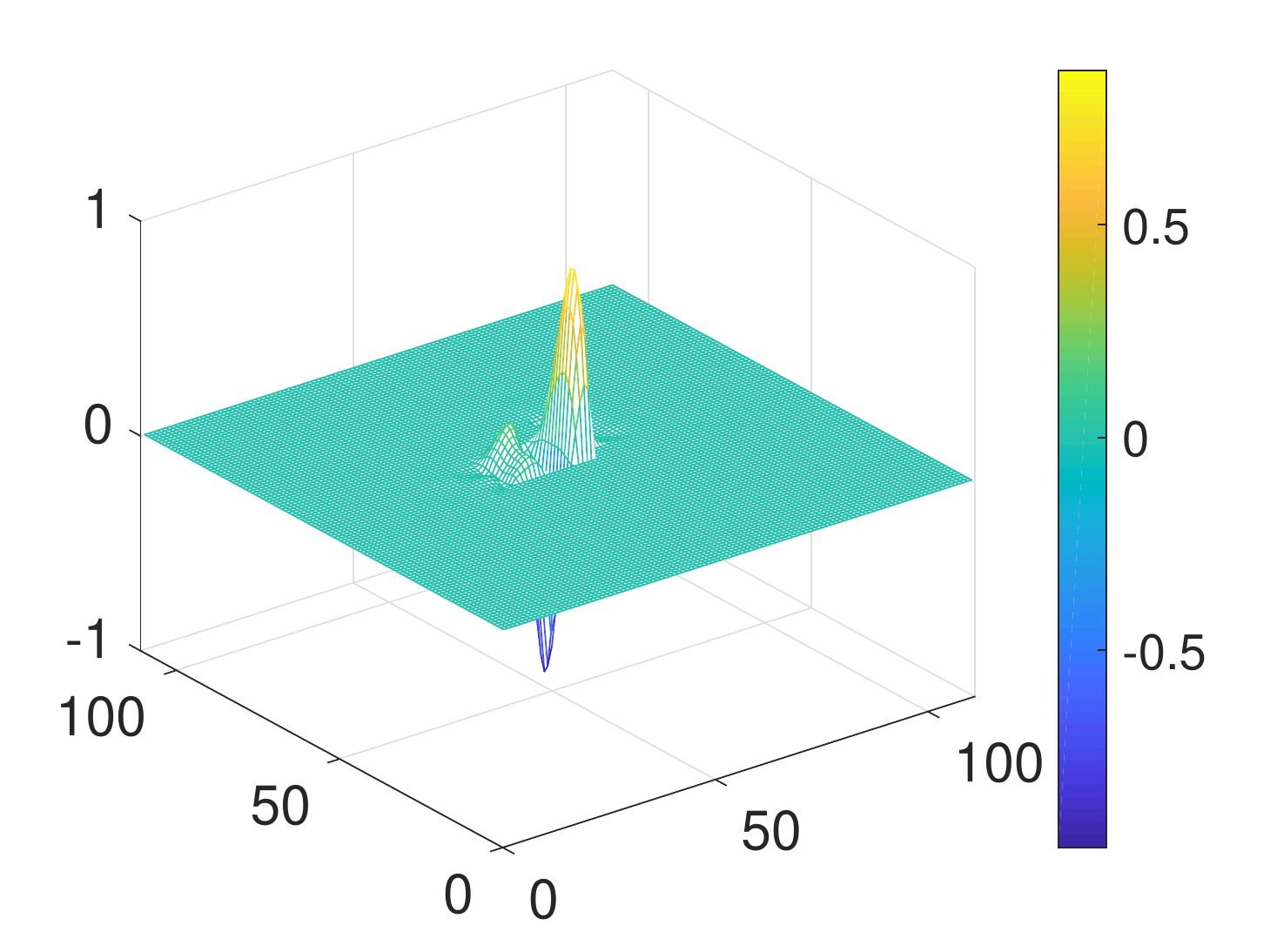}
  \caption{An example of  constrained CEM-GMsFEM basis function. First component (left),  and second component (right).}
  \label{example2msbasisconstraint} 
\end{figure}
\section{Convergence Analysis}\label{section4}

In this section, the stability and convergence analysis of our method  are carried out. 
We will first give an outline of our analysis. In Lemma~\ref{lem:coer}, we will prove the coercivity and continuity
of the bilinear form of $a_{\text{DG}}$ in our scheme (\ref{eq:ms}). In Lemma~\ref{lemma:dgcoupling}, we will prove 
that the global basis functions have an exponential decay property. In particular, we will show that the difference between
the global basis function and the corresponding local basis function is small if the oversampling region is sufficiently large.
In Lemma~\ref{lemma4.4}, we prove a stable decomposition property.
Theorem~\ref{thm:thm1} is the main result, which states that the error of approximation in the DG norm is bounded above by $O(H)$
if the size of the oversampling region is proportional to the logarithm of the coarse grid size $H$ and the logarithm of the contrast ratio of the coefficient.
In addition, Lemma~\ref{lemma4.5} is an analog of the Lemma~\ref{lemma:dgcoupling} for the relaxed method.
Finally, Theorem~\ref{thm:2} is the main convergence result for the relaxed method.

For our analysis, we define the following DG norm
\begin{equation}
    \left \|  v  \right \|_{\text{DG}}^{2}=\sum_{K\in\mathcal{T} ^{H}}\int_{K} \epsilon \left ( v  \right ):C:\epsilon \left ( v  \right ) \text{dx}+\frac{\gamma }{h}\sum_{E\in \mathcal{E}^{H}}\int_{E} \left(\underline{{[\![ v]\!] }}:C:\underline{{[\![ v]\!] } }+[\![ v]\!] \cdot  D \cdot [\![ v]\!]\right)\text{ds}. 
\end{equation}
The following constants are independent of the mesh size and the number of oversampling layers unless otherwise indicated. We assume that  there exist positive constants $0<c_{m_0} \leq c_{m_1}$ such that for a.e. $x \in \Omega,$  $ C\left(x\right)$ is a positive definite matrix with $$ c_{m_0} \leq \lambda_{\min }(C\left(x\right)) \leq \lambda_{\max }\left(C\left(x\right)\right)  \leq c_{m_1} ,$$
where $\lambda_{\min }(C\left(x\right))$ and $\lambda_{\max }\left(C\left(x\right)\right)$ are the minimum and the maximum eigenvalues of $C\left(x\right)$.

We also define the $b$-norm as $\|\cdot\|_{b}:=\sqrt{b\left(\cdot, \cdot\right)}$.  Moreover, for $\lambda_i^j$ in (\ref{abf}),  we let $$c_1= \min _{1\leq j\leq N_c}\lambda_{G_j +1}^{j}, \ c_2=\max _{1 \leq j \leq N_c} \max _{1 \leq i \leq G_{j}} \lambda_{i}^{j}.$$

In the following, we will show the continuity and coercivity of the IPDG bilinear form. For a coarse element $K$, let $n_{\partial K}$ be the unit outward normal vector on $\partial K$.
Moreover, we define $V_h\left(\partial K\right)$ as the restriction of $V_h\left(K\right)$ to $\partial K$.
Given $u \in V_h\left(K\right)$,
the normal flux $\sigma \left(u\right)n_{\partial K}$ belongs to $V_{h}\left(\partial K\right)$ and is defined by the following formula:
\begin{equation}
\label{eq:flux}
    \int_{\partial K}\left ( \sigma\left ( u \right )n_{\partial K} \right )\cdot v \ \text{ds}=\int _{K} \epsilon \left ( u \right ):C:\epsilon \left ( \hat{v} \right )\text{dx}, \  \hat{v}\in V_{h}\left(\partial K\right),
\end{equation}
where $\hat{v} \in K$ is an extension of $v$ to $K$ 
defined by solving 
\begin{equation*}
\int_{K}\varepsilon \left ( \hat{v} \right ) :C:\varepsilon \left ( w \right ) = 0, \ \forall \ w\in V_h\left(K\right)
\end{equation*}
with the boundary condition $\hat{v}=v$ on $\partial K$.
Using the results in \cite{chung2014generalizedwave}, there is a uniform constant $A>0$ such that 
\begin{equation}
\label{eq:trace}
    \sum_{K\in \mathcal{T}^{H}}\int_{\partial K}\left (  \sigma\left ( v \right )  n_{\partial K}  \right)^{2}\text{ds}\leq \frac{Ac_{m_1}^{2}}{2c_{m_0}h}\left (\sum_{K\in \mathcal{T}^{H}}\int_{K}\varepsilon \left ( v \right ):C:\varepsilon \left ( v \right )\text{dx} \right ).
\end{equation}

We will employ the following coercivity and continuity results on the IPDG bilinear form in our study, assuming the penalty value $\gamma$ is suitably large.
\begin{lemma}\label{lem:coer}
Assume that the penalty parameter $\gamma$ satisfies $\gamma>Ac_{m_1}^2 c_{m_0}^{-2}$, where the constant $A>0$ is defined in (\ref{eq:trace}).
Then the bilinear form $a_{\text{DG}}$ defined in (\ref{adg}) is continuous and coercive, that is
\begin{equation}
      \left |a_{\text{DG}}\left (  u, v \right )\right|  \leq 2\|u\|_{\text{DG}}\|v\|_{\text{DG}}, \ \forall\ u, v\in V_{h},\label{equ:lemma4.1-1}
\end{equation}
and 
\begin{equation}
     \frac{1}{2}\|v\|_{\text{DG}}^{2} \leq a_{\text{DG}}\left(v, v\right) , \ \forall\ v \in V_{h}.\label{equ:lemma4.1-2}
\end{equation}
\end{lemma}

\begin{proof}
By the definition of $a_{\text{DG}}\left (  v, v \right ),$ we have 
\begin{equation*}
\begin{aligned}
a_{\text{DG}}\left (  v, v \right )&=\sum_{K\in \mathcal{T} ^{H}}\int_{K} \sigma \left (  v \right ):\epsilon \left ( v \right )\text{dx}\\
&-\sum_{E\in \mathcal{E}^{H}}\int_{E}\left ( \left \{ \sigma\left( v\right) \right \} :\underline{{{[\![ v]\!]}} }+ \underline{{{[\![v]\!]}}}:  \left \{  \sigma\left( v\right) \right \} \right)\text{ds}\\
&+\sum_{E\in \mathcal{E} ^{H}}\frac{\gamma }{h}\int_{E}\left (\underline{{{[\![v]\!]}}}:\left \{  C \right \}:\underline{{{[\![ v]\!]}}} +{[\![ v]\!]} \cdot \left \{ D\right \} \cdot {[\![ v]\!]}\right )\text{ds},
\end{aligned}
\end{equation*}
and by the definition of the $\text{DG}$-norm, we have
\begin{equation}
\label{eq:adg}
    a_{\text{DG}}\left (  v, v \right )=\left \|  v  \right \|_{\text{DG}}^{2}-\sum_{E\in \mathcal{E}^{H}}\int_{E}\left ( \left \{ \sigma\left( v\right) \right \} :\underline{{{[\![ v]\!]}} }+ \underline{{{[\![v]\!]}}}:  \left \{  \sigma\left( v\right) \right \} \right)\text{ds}.
\end{equation}
To estimate the second term in the right hand side of (\ref{eq:adg}), we note that
$$\sum_{E\in \mathcal{E}^{H}}\int_{E}\left ( \left \{ \sigma\left( v\right) \right \} :\underline{{{[\![ v]\!]}} }+ \underline{{{[\![v]\!]}}}:  \left \{  \sigma\left( v\right) \right \} \right)\text{ds} = 2\sum_{E\in \mathcal{E}^{H}}\int_{E}\left ( \left \{ \sigma\left( v\right) \right \} :\underline{{{[\![ v]\!]}} } \right)\text{ds}.$$
Then by the Cauchy-Schwarz inequality, we obtain
\begin{equation*}
    \begin{aligned}
    2\sum_{E\in \mathcal{E}^{H}}\int_{E}\left ( \left \{ \sigma\left( v\right) \right \} :\underline{{{[\![ v]\!]}} } \right)\text{ds}
    &\leq \frac{2h}{\gamma c_{m_0}}\sum_{E\in \mathcal{E}^{H}}\int_{E}\left ( \left \{ \sigma(v) n \right \}  \right)^{2}\text{ds}+\frac{c_{m_0}}{2}\sum_{E\in \mathcal{E}^{H}}\int_{E}\frac{\gamma }{h}\left ( \underline{{{[\![ v]\!]}} } \right)^{2}\text{ds}\\
    & \leq   \frac{2h}{\gamma c_{m_0}}\sum_{K\in \mathcal{T}^{H}}\int_{\partial K}\left (  \sigma(v) n_{\partial K}  \right)^{2}\text{dx}+\frac{1}{2}\sum_{E\in \mathcal{E}^{H}}\int_{E}\frac{\gamma }{h}\left ( \underline{{{[\![ v]\!]}} }:C:\underline{{{[\![ v]\!]}} } \right)\text{ds}.
    \end{aligned}
\end{equation*}
With (\ref{eq:trace}), it is straightforward to get 
\begin{equation*}
    \begin{aligned}
    2\sum_{E\in \mathcal{E}^{H}}\int_{E}\left ( \left \{ \sigma\left( v\right) \right \} :\underline{{{[\![ v]\!]}} } \right)\text{ds}
    & \leq  \frac{Ac_{m_1}^{2}}{\gamma c_{m_0}}\left (\sum_{E\in \mathcal{E}^{H}}\int_{E}\varepsilon \left ( v \right ):C:\varepsilon \left ( v \right )\cdot n_{\partial K}\text{ds} \right )\\
    &\quad +\frac{1}{2}\sum_{E\in \mathcal{E}^{H}}\int_{E}\frac{\gamma }{h}\left ( \underline{{{[\![ v]\!]}} }:C:\underline{{{[\![ v]\!]}} } \right)\text{ds}.
    \end{aligned}
\end{equation*}
Therefore, if we take $\gamma \geq \frac{A c_{m_1}^{2}}{ c_{m_0}^{2}},$ we obtain
\begin{equation*}
    \begin{aligned}
    a_{\text{DG}}\left (  v, v \right )
    &=\left \|  v  \right \|_{\text{DG}}^{2}-2\sum_{E\in \mathcal{E}^{H}}\int_{E}\left ( \left \{ \sigma\left( v\right) \right \} :\underline{{{[\![ v]\!]}} } \right)\text{ds}\\
    &\geq \frac{1}{2}\left \| v \right \|_{\text{DG}}, \quad \forall \ v \in V_{h}.
    \end{aligned}
\end{equation*}
We hence proved (\ref{equ:lemma4.1-2}).

(\ref{equ:lemma4.1-1}) can be proved in a similar way.  By the definition of $a_{\text{DG}}$ and the properties of inequalities, we obtain
\begin{equation*}
    \begin{aligned}
    \left |a_{\text{DG}}\left (  u, v \right )\right| 
      &\leq  \underset{S_{1}}{\underbrace{\left |\sum_{K\in \mathcal{T}^{H}}\int_{K} \sigma \left (  u \right ):\epsilon \left ( v \right )\text{dx} \right |} }\\
      &\quad +\underset{S_{2}}{\underbrace{\left |\sum_{E\in \mathcal{E}^{H}}\int_{E}\left (\left\{\sigma\left( u\right) \right\} :\underline{{[\![ v]\!]}}+\underline{{{[\![u]\!]}}}: \left\{\sigma\left( v\right) \right\} \right)\text{ds} \right |}} \\
      &\quad+ \underset{S_{3}}{\underbrace{\left | \sum_{E\in \mathcal{E}^{H}}\frac{\gamma }{h}\int_{E}\left (\underline{ {{[\![ u]\!]}}}:\left\{C \right\}:\underline{{{[\![ v]\!]}} }+{[\![ u]\!]}\cdot \left\{D\right\}\cdot {[\![ v]\!]}\right )\text{ds} \right |}} . \\
    \end{aligned}
\end{equation*}
Notice that 
\begin{equation*}
    \begin{aligned}
    S_{1}& \leq\sum_{K\in \mathcal{T}^{H}}\left |\int_{K} \sigma \left (  u \right ):\epsilon \left ( v \right )\text{dx} \right |,\\
    S_{2}&\leq \sum_{E\in \mathcal{E}^{H}}\left |\int_{E}\left (\left\{\sigma\left( u\right) \right\} :\underline{{[\![ v]\!]}}+\underline{{{[\![u]\!]}}}: \left\{\sigma\left( v\right) \right\} \right)\text{ds} \right |, \\
    S_{3} &\leq  \sum_{E\in \mathcal{E}^{H}}\frac{\gamma }{h}\left |\int_{E}\left (\underline{ {{[\![ u]\!]}}}:\left\{C \right\}:\underline{{{[\![ v]\!]}} }+{[\![ u]\!]}\cdot \left\{D\right\}\cdot {[\![ v]\!]}\right )\text{ds}\right |  .
    \end{aligned}
\end{equation*}
The first term $S_{1}$ is easily estimated as follows:
\begin{equation*}
    S_{1}\leq \left ( \sum_{K\in \mathcal{T}^{H}} \int_{K}\epsilon \left ( u \right ):C:\epsilon \left ( u \right )\text{dx}\right )^{\frac{1}{2}}\left ( \sum_{K\in \mathcal{T}^{H}}\int_{K} \epsilon \left ( v \right ):C:\epsilon \left ( v \right )\text{dx}\right )^{\frac{1}{2}}.
\end{equation*}
By the arithmetic-geometric mean inequality, the third term $S_{3}$ is also easily estimated as follows:
\begin{equation*}
\begin{aligned}
    S_{3}&\leq  \left ( \sum_{E\in \mathcal{E}^{H}}\frac{\gamma }{h}\int_{E}\left (\underline{ {{[\![ u]\!]}}}:\left\{C \right\}:\underline{{{[\![ u]\!]}} }\right ) \text{ds}\right )^{\frac{1}{2}}\left ( \sum_{E\in \mathcal{E}^{H}}\frac{\gamma }{h}\int_{E}\left (\underline{ {{[\![ v]\!]}}}:\left\{C \right\}:\underline{{{[\![ v]\!]}} }\right )  \text{ds}\right )^{\frac{1}{2}}\\
    &\quad + \left ( \sum_{E\in \mathcal{E}^{H}}\frac{\gamma }{h}\int_{E}\left ( {{[\![ u]\!]}}:\left\{D \right\}:{{[\![ u]\!]} }\right ) \text{ds}\right )^{\frac{1}{2}}\left ( \sum_{E\in \mathcal{E}^{H}}\frac{\gamma }{h}\int_{E}\left ({{[\![ v]\!]}}\cdot \left\{D \right\}\cdot {{[\![ v]\!]} }\right )  \text{ds}\right )^{\frac{1}{2}}\\
    & \leq \left ( \sum_{E\in \mathcal{E}^{H}}\frac{\gamma }{h}\int_{E}\left (\underline{ {{[\![ u]\!]}}}:\left\{C \right\}:\underline{{{[\![ u]\!]}} } +{{[\![ u]\!]}}\cdot \left\{D \right\}\cdot {{[\![ u]\!]} }\right ) \text{ds}\right )^{\frac{1}{2}}\\
    &\quad \times \left ( \sum_{E\in \mathcal{E}^{H}}\frac{\gamma }{h}\int_{E}\left (\underline{ {{[\![ v]\!]}}}:\left\{C \right\}:\underline{{{[\![ v]\!]}} }+{{[\![ v]\!]}}\cdot \left\{D \right\}\cdot{{[\![ v]\!]} }\right )  \text{ds}\right )^{\frac{1}{2}}.
\end{aligned}
\end{equation*}
For $S_{2},$ again using the Cauchy-Schwarz inequality and  (\ref{eq:trace}),  we can estimate it as follows:
\begin{equation*}
    \begin{aligned}
    S_{2} &\leq  \left (\frac{h}{\gamma } \sum_{E\in \mathcal{E}^{H}}\int_{E}  \left ( \left \{ \sigma \left (  u \right )n\right \} \right )^{2}    \text{ds} \right )^{\frac{1}{2}}\left (\sum_{E\in \mathcal{E}^{H}}\frac{\gamma }{h}\int_{E}\underline{{{[\![ v]\!]}} }:\underline{{{[\![ v]\!]}} }\text{ds}\right )^{\frac{1}{2}}\\
    &\quad +\left (\frac{h}{\gamma } \sum_{E\in \mathcal{E}^{H}}\int_{E}  \left ( \left \{ \sigma \left (  v\right )n\right \} \right )^{2}   \text{ds} \right )^{\frac{1}{2}}\left (\sum_{E\in \mathcal{E}^{H}}\frac{\gamma }{h}\int_{E}\underline{{{[\![ u]\!]}} }:\underline{{{[\![ u]\!]}} }\text{ds}\right )^{\frac{1}{2}}\\
    &\leq \left (\frac{h}{\gamma c_{m_0}} \sum_{K\in \mathcal{T}^{H}}\int_{\partial K}\left ( \sigma \left (  u\right ) n_{\partial K} \right )^{2}\text{dx} \right )^{\frac{1}{2}}\left (\sum_{E\in \mathcal{E}^{H}}\frac{\gamma }{h} \int_{E}\underline{{{[\![ v]\!]}} }:C:\underline{{{[\![ v]\!]}} }\text{ds}\right )^{\frac{1}{2}}\\
    &\quad +\left (\frac{h}{\gamma c_{m_0}} \sum_{K\in \mathcal{T}^{H}}\int_{\partial K}\left ( \sigma \left (  v \right ) n_{\partial K} \right )^{2}\text{dx} \right )^{\frac{1}{2}}\left (\sum_{E\in \mathcal{E}^{H}}\frac{\gamma }{h} \int_{E}\underline{{{[\![ u]\!]}} }: C:\underline{{{[\![ u]\!]}} }\text{ds}\right )^{\frac{1}{2}}\\
    &\leq \left (\frac{Ac_{m_1}^{2}}{\gamma c_{m_0}^{2}} \right )^{\frac{1}{2}}\left ( \sum_{K\in \mathcal{T}^{H}}\int_{ K}\varepsilon \left ( u \right ):C:\varepsilon \left ( u \right )\text{dx}  \right )^{\frac{1}{2}}\left (\sum_{E\in \mathcal{E}^{H}}\frac{\gamma }{h} \int_{E}\underline{{{[\![ v]\!]}} }:C:\underline{{{[\![ v]\!]}} }\text{ds}\right )^{\frac{1}{2}}\\
   &\quad + \left (\frac{Ac_{m_1}^{2}}{\gamma c_{m_0}^{2}} \right )^{\frac{1}{2}}\left ( \sum_{K\in \mathcal{T}^{H}}\int_{ K}\varepsilon \left ( v \right ):C:\varepsilon \left ( v \right )\text{dx}  \right )^{\frac{1}{2}}\left (\sum_{E\in \mathcal{E}^{H}}\frac{\gamma }{h} \int_{E}\underline{{{[\![ u]\!]}} }:C:\underline{{{[\![ u]\!]}} }\text{ds}\right )^{\frac{1}{2}}.
    \end{aligned}
\end{equation*}
Then collecting above estimations, if $\gamma \geq \frac{A c_{m_1}^{2}}{ c_{m_0}^{2}}$ , we have $$ \left|a_{\text{DG}}\left(u, v\right)\right| \leq 2\|u\|_{\text{DG}}\|v\|_{\text{DG}}.$$ (\ref{equ:lemma4.1-1}) has been proved.
\end{proof}

To conclude, we must prove that the global basis functions are indeed localizable.
For each coarse block $K$, $B$ is a bubble function satisfies
\begin{equation*}
\begin{split}
& B\left (x \right )>0,\  \forall \  x \in \text{int}\left ( K \right ), \\
& B\left ( x \right )=0,\ \forall \  x \in \partial K. 
\end{split}
\end{equation*}
We  take $B=\Pi_{j} \chi_{j}$ where the product is taken over all vertices $j$ on $\partial K$ and define the  constant
\begin{equation*}
C_{\pi}=\sup _{K \in \mathcal{T}^{H}, \mu \in W_{\text {aux }}} \frac{\int_{K} \tilde{\kappa} \mu^{2}}{\int_{K} B \tilde{\kappa} \mu^{2}}.
\end{equation*}
In the following, we will prove a technical result in Lemma~\ref{lemmasupp}. This is a surjectivity result of the projection operator $\pi$.
It says that for any function $\psi$ in the auxiliary space $W_{\text{aux}}$, we can find a function $\phi$ in the DG space $V_h$ such that the DG norm of the function $\phi$ is controlled by the norm of the function $\psi$. Also, the support of the function $\phi$ is contained in the support of the function $\psi$.

\begin{lemma}\label{lemmasupp}
Given $\psi \in W_{\text{aux}} $, we could find $\phi \in V_{h}$, such that
$$\pi \left(\phi  \right)=\psi , \quad \left \| \phi  \right \|_{\text{DG}}^{2}\leq M\left \| \psi  \right \|_{b}^{2}, \quad \text{supp}\left ( \phi  \right )\subset \text{supp}\left ( \psi  \right ) .$$
We write $M=C_{T}C_{\pi}\left(1+c_{2}\right),$ where $C_{T}$ is the maximum number of vertices over all coarse elements.
\end{lemma}

\begin{proof}
Let $\psi \in W_{\text{aux}}\left ( K_{j} \right ).$ Then the solution of (\ref{minimization}) is found by the following steps.
First, we find $\phi \in V_{h}\left ( K_{j} \right ), {\psi }'\in W_{\text{aux}}\left ( K_{j} \right )$ such that
\begin{equation}
\begin{aligned}
 a_{j}\left ( \phi ,w \right )+b_{j}\left ( w,{\psi }' \right )&=0,\quad\forall\  w\in V_{h}\left ( K_{j} \right ),\quad\\
b_{j}\left ( \phi ,\eta  \right )&=b_{j}\left (\psi  ,\eta  \right ),\quad\forall \ \eta \in W_{\text{aux}}\left ( K_{j} \right ).\label{equ:argmin}
\end{aligned}
\end{equation}
Note that the existence of (\ref{minimization}) is equivalent to find $\phi \in V_{h}\left(K_{j}\right)$ such that
   $$ b_{j}\left(\phi, \psi\right) \geq   \beta     \left \| \psi  \right \|_{b\left ( K_{j} \right )}^{2}, \quad \left \| w \right \|_{a_{\rm DG}\left ( K_{j} \right )}\geq     \beta      \left \| \psi  \right \|_{b\left ( K_{j} \right )}^{2},$$
where $\beta $  is a constant to be determined.

Note that $\psi$ is supported in $K_{j}$. Let $B$ is a bubble function and $\phi=B\psi$, then we obtain
$$b_{j}\left ( \phi ,\psi  \right )\geq C_{\pi}^{-1}\left \| \psi  \right \|_{b\left ( K_{j} \right )}^{2}.$$
Since
$$\nabla \left ( B\psi  \right )=\psi \nabla B+B\nabla \psi ,\quad |B|\leq 1,\quad |\nabla B|^{2}\leq C_{T}\sum_{j}|\nabla \chi _{j}|^{2},$$
we have 
\begin{equation}
    \left \| \phi  \right \|_{a_{\text{DG}}\left ( K_{j} \right )}^{2}=\left \| B\psi  \right \|_{a_{\text{DG}}\left ( K_{j} \right )}^{2}\leq C_{T} C_{\pi}\left \| \phi  \right \|_{a_{\text{DG}}\left ( K_{j} \right )}\left ( \left \| \phi  \right \|_{a_{\text{DG}}\left ( K_{j} \right )}+\left \| \phi  \right \|_{b\left ( K_{j} \right )} \right ).
\end{equation}
Finally, using the (\ref{abf}), we note that
\begin{equation*}
\left \| \psi  \right \|_{a_{\text{DG}}\left ( K_{j} \right )}\leq \left ( \underset{1\leq i\leq G_{j}}{\text{max}} \lambda _{i}^{j}\right )\left \| \phi  \right \|_{b\left ( K_{j} \right )}.
\end{equation*}
This proves that (\ref{minimization}) has a unique solution. We see that $\phi$ and $\psi$ also satisfy (\ref{equ:argmin}) respectively.  Hence we are done to prove this lemma.
\end{proof}

Next we will prove the convergence of the method.
\subsection{Convergence Analysis}
In this part, we prove the convergence of our DG coupling method using the multiscale basis functions defined in (\ref{minimization}). 
First of all, in Lemma~\autoref{lemma:dgcoupling}   we will prove a localization result of the global multiscale basis functions defined in   (\ref{basisiob}).
In the result shown in (\ref{eq:local1}), we see that the difference between the local basis function and the global basis function 
is less than a constant $M_1$, which decays exponentially as the number of oversampling size $p$ increases.

\begin{lemma} \label{lemma:dgcoupling}
We consider an oversampling area $K_{j,p}$. That is, $K_{j,p}$ is obtained by extending  $K_{j}$ by $p$ coarse grid layers. Let $ \psi _{i}^{j}\in W_{\text{aux}}$ be the auxiliary multiscale basis function, $\phi _{i,\text{glo}}^{j} $ be the solution of   (\ref{basisiob}), and $\phi_{i,\text{ms}}^{j} $ be the solution of (\ref{minimization}). Then we have
\begin{equation}
      \left \| \phi _{i,\text{ms}}^{j}- \phi _{i,\text{glo}}^{j} \right \|_{\text{DG}}\leq M_{1}\left \| \psi _{i}^{j} \right \|_{b\left ( K_{j} \right )}^{2}, \label{eq:local1}
 \end{equation}
 where $M_{1}=4M \left ( c_{1}+\frac{1}{c_{1}} \right )  \left [ 1+\left ( \frac{M}{2}+c_{1}^{\frac{1}{2}} \right )^{-1} \right ] ^{\frac{1}{2}-p}.$ 
\end{lemma}
\begin{proof}
By using the solution of the minimization problems (\ref{basisiob}) and (\ref{minimization}), we can get
\begin{equation}
    a_{\text{DG}}\left ( \phi _{i,\text{ms}}^{j},w \right )+b\left (w, \left ({\psi }' \right ) _{i}^{j} \right)=0,\quad\forall \ w\in V_{h},
\end{equation}
and
\begin{equation}
    a_{\text{DG}}\left (  \phi _{i,\text{glo}}^{j},w \right )+b\left (w, \left ({\psi }' \right ) _{i,\text{glo}}^{j} \right )=0,\quad \forall\  w\in V_{h}\left ( K_{j,p} \right ),
\end{equation}
where $\left ({\psi }' \right ) _{i}^{j},\left ({\psi }' \right ) _{i,\text{glo}}^{j}\in W_{\text{aux}}.$
Substracting the above two equations and restricting $w\in V_{h}\left ( K_{j,p} \right ),$ we have $$ a_{\text{DG}}\left ( \phi _{i,\text{ms}}^{j},w \right )-a_{\text{DG}}\left (  \phi _{i,\text{glo}}^{j},w \right )+b\left (w, \left ({\psi }' \right ) _{i}^{j} \right)-b\left (w, \left ({\psi }' \right ) _{i,\text{glo}}^{j} \right )=0,$$
$$a_{\text{DG}}\left ( \phi _{i,\text{ms}}^{j}-  \phi _{i,\text{glo}}^{j},w \right )=0,\ \forall \ w\in V_{h}\left ( K_{j,p} \right ).$$

For $\psi_{i}^{j}\in W_{\text{aux}}$, using Lemma \ref{lemmasupp}, there exists a function $\widetilde{\phi} _{i}^{j}\in V_{h}$ such that 
$$\pi\left ( \widetilde{\phi} _{i}^{j} \right )=\psi_{i}^{j}, \quad \left \| \widetilde{\phi}_{i}^{j}  \right \|_{\text{DG}}^{2}\leq M\left \| \psi_{i}^{j} \right \|_{b}^{2}, \quad \text{supp}\left (\widetilde{\phi} _{i}^{j}  \right )\subset K_{j}.$$
Let $g=\phi _{i,\text{ms}}^{j}-\widetilde{\phi}_{i}^{j}.$ Note that $g\in \widetilde{V_{h}}$ and $\pi \left(g\right)=0.$
Here, we have $$V_{h}\left ( K_{j,p} \right )=H_{0}^{1}\left ( K_{j,p} \right ),\quad
\widetilde{V_{h}}=\left \{ w\in V_{h}|\pi\left ( w \right )=0 \right \}.$$
Therefore, for $l=  \phi _{i,\text{glo}}^{j}-\widetilde{\phi} _{i}^{j}   \in \widetilde{V_{h}},$ we have 
\begin{equation*}
\begin{aligned}
  \left \| \phi _{i,\text{ms}}^{j}- \phi _{i,\text{glo}}^{j} \right \|_{\text{DG}}^{2}&= a_{\text{DG}}\left ( \phi _{i,\text{ms}}^{j}- \phi _{i,\text{glo}}^{j},\phi _{i,\text{ms}}^{j}- \phi _{i,\text{glo}}^{j} \right )
 \\&= a_{\text{DG}}\left ( \phi _{i,\text{ms}}^{j}- \phi _{i,\text{glo}}^{j},\phi _{i,\text{ms}}^{j}-\widetilde{\phi} _{i}^{j}-  \phi _{i,\text{glo}}^{j}+\widetilde{\phi} _{i}^{j}\right )
 \\&=a_{\text{DG}}\left ( \phi _{i,\text{ms}}^{j}- \phi _{i,\text{glo}}^{j},g-l \right ) .
\end{aligned}
\end{equation*}
So, we obtain 
\begin{equation}
\left \| \phi_{i,\text{ms}}^{j}- \phi _{i,\text{glo}}^{j} \right \|_{\text{DG}}\leq \left\| g-l \right \|_{\text{DG}}.
\end{equation}

We will now estimate $\left \| \phi _{i,\text{ms}}^{j}- \phi _{i,\text{glo}}^{j} \right \|_{\text{DG}}$, that is, the difference between global multiscale basis functions $\phi _{i,\text{glo}}^{j}$  and localized multiscale basis functions $\phi _{i,\text{ms}}^{j}$. We will observe that the global multiscale basis functions decay and have tiny values outside of a sufficiently large oversampled area. In our proof, we shall employ a cutoff function.  Define a function
\begin{equation}
\begin{aligned}
\chi _{j}^{p+i,p}&=1 \quad \text{in} \quad K_{j,p},\\ 
\chi _{j}^{p+i,p}&=0 \quad \text{in} \quad \Omega \setminus  K_{j,p}.
\end{aligned}
\end{equation}
where $j>0,K_{j,p}\subset K_{j,p+1}.$ Since $g\in \widetilde{V_{h}},$ we have
$$b_{i}\left ( \chi _{j}^{p+1,p}g,\psi _{n}^{i} \right )=b_{i}\left ( g,\psi _{n}^{i} \right )=0, \ \forall\ n=1,2,...N.$$
Then $$\text{supp}\left ( \pi\left ( \chi _{j}^{p+1,p}g \right ) \right )\subset K_{j,p+1}\setminus K_{j,p}.$$
Using Lemma \ref{lemmasupp}, for $\pi\left ( \chi _{j}^{p+1,p}g \right )$, $ \exists \ {\psi }'\in V_h$, such that $$ \text{supp}\left ( {\psi }' \right )\subset K_{j,p+1}\setminus K_{j,p}, \quad
\pi\left ( {\psi }'-\chi _{j}^{p+1,p}g \right )=0,$$
and
\begin{equation}
\left \| {\psi }' \right \|_{a_{\text{DG}}\left ( K_{j,p+1}\setminus K_{j,p} \right )}
\leq \sqrt{M}\left \| \pi\left ( \chi _{j}^{p+1,p}g \right ) \right \|_{b\left ( K_{j,p+1}\setminus K_{j,p} \right )}.\label{equ:44}
\end{equation}
Remember that $\pi$ is a projection, restricting $\left \| \psi _{i}^{j} \right \|\leq 1$, then we note that
\begin{equation}
\begin{aligned}
\left \| \pi\left ( \chi _{j}^{p+1,p}g \right ) \right \|_{b\left ( K_{j,p+1}\setminus K_{j,p} \right )}
&=\left \| \sum_{j=1}^{N}\sum_{i=1}^{G_{j}}b_{j}\left ( \chi_{j} ^{p+1,p}g,\psi _{i}^{j}\right )\psi _{i}^{j} \right \|_{b\left ( K_{j,p+1}\setminus K_{j,p} \right )}\\
&\leq \left \| \chi _{j}^{p+1,p}g  \right \|_{b\left ( K_{j,p+1}\setminus K_{j,p} \right )}.\label{equ:45}
\end{aligned}
\end{equation}
Combining (\ref{equ:44}) and (\ref{equ:45}),  that is 
\begin{equation}
\left \| {\psi }' \right \|_{a_{\text{DG}}\left ( K_{j,p+1}\setminus K_{j,p} \right )}
\leq \sqrt{M}\left \| \chi _{j}^{p+1,p}g  \right \|_{b\left ( K_{j,p+1}\setminus K_{j,p} \right )}.
\end{equation}

Hence, let  $l={\psi }'+\chi _{j}^{p+1,p}g,$ then we have
\begin{equation}
\begin{aligned}
\left \| \phi _{i,\text{ms}}^{j}- \phi _{i,\text{glo}}^{j} \right \|_{\text{DG}}
&\leq \left \| g-l \right \|_{\text{DG}}\\
&\leq \left \| \left ( 1-\chi _{j}^{p+1,p} \right )g \right \|_{\text{DG}}+\left \| {\psi }' \right \|_{a_{\text{DG}}\left ( K_{j,p+1}\setminus K_{j,p} \right )}.\label{equ:1}
\end{aligned}
\end{equation}

Subsequently, we will estimate (\ref{equ:1}) in three parts, first estimating $\left \| \left ( 1-\chi _{j}^{p+1,p} \right )g \right \|_{\text{DG}}$ and $\left \| {\psi }' \right \|_{a_{\text{DG}} \left ( K_{j,p+1}\setminus K_{j,p} \right )}$ separately, then further estimating the first part's results in the second part, and finally obtaining our final estimate of the value between  $\phi _{i,\text{ms}}^{j}$ and $\phi _{i,\text{glo}}^{j} $ in the third part.\\
$\mathbf{Part1.}$ Next we estimate $\left \| \left ( 1-\chi _{j}^{p+1,p} \right )g \right \|_{\text{DG}}$.  By using $\nabla\left(\left(1-\chi_{j}^{p, p-1}\right) g\right)=-\nabla \chi_{j}^{p+1,p} g+\left(1-\chi_{j}^{p+1,p}\right) \nabla g$ and $0 \leq 1-\chi_{j}^{p+1,p} \leq 1$, we can get
\begin{equation}
\begin{aligned}
\left \| \left ( 1-\chi _{j}^{p+1,p} \right )g \right \|_{\text{DG}}^{2}
&\leq 2\left(\|g\|_{a_{D G\left(\Omega \backslash K_{j, p}\right)}}^{2}+\|g\|_{b\left(\Omega \backslash K_{j,p}\right)}^{2}\right) \\
&\leq 2\left ( 1+\frac{1}{c_{1}^{2}} \right )\|g\|_{a_{\text{DG}}\left ( \Omega \setminus K_{j,p} \right )}^{2}\\
&\leq 2\left ( c_{1}+\frac{1}{c_{1}} \right )^{2}\|g\|_{a_{\text{DG}}\left ( \Omega \setminus K_{j,p} \right )}^{2}
\end{aligned}
\end{equation}

We will estimate the second term $\left \| {\psi }' \right \|_{a_{\text{DG}} \left ( K_{j,p+1}\setminus K_{j,p} \right )}$. Since $\chi _{j}^{p+1,p}\leq 1,$  we have
\begin{equation}
\begin{aligned}
 \left \| {\psi }' \right \|_{a_{\text{DG}} \left ( K_{j,p+1}\setminus k_{j,p} \right )}
 &\leq \sqrt{M} \left \| \chi _{j}^{p+1,p}g \right \|_{b\left ( K_{j,p+1}\setminus K_{j,p} \right )}\\
 &\leq\frac{ \sqrt{M} }{c_{1}}\left \| g\right \|_{a_{\text{DG}}\left ( K_{j,p+1}\setminus K_{j,p} \right )}.\label{equ:3}
\end{aligned}
\end{equation}

Combining the estimates (\ref{equ:1})-(\ref{equ:3}), we obtain
\begin{equation}
\left \| \phi _{i,\text{ms}}^{j}- \phi _{i,\text{glo}}^{j} \right \|_{\text{DG}}  \leq 2\sqrt{M} \left ( c_{1}+\frac{1}{c_{1}} \right )\left \| g \right \|_{a_{\text{DG}}\left ( \Omega \setminus K_{j,p} \right )}\label{part1}.  
\end{equation}
$\mathbf{Part2.}$ Then we estimate $\left \| g \right \|_{a_{\text{DG}}\left ( \Omega \setminus K_{j,p} \right )}^{2}$. We will use the recursion method, so we first need to prove that the following formula
\begin{equation}
   \left \| g \right \|_{a_{\text{DG}}\left ( \Omega \setminus K_{j,p} \right )}^{2}\leq \left [ 1-\left ( 1+\frac{M}{2}+c_{1}^{\frac{1}{2}} \right )^{-1} \right ]\left \| g \right \|_{a_{\text{DG}}\left ( \Omega \setminus K_{j,p-1} \right )}^{2}\label{part2}.
\end{equation}
Let $\varsigma =1-\chi _{j}^{p,p-1}$. Then we obtain
$$\varsigma \equiv 1 \quad \text{in} \ \Omega \setminus K_{j,p}, \quad \varsigma \equiv 0\ \text{otherwise}.$$
Hence, we have
\begin{equation}
\begin{aligned}
 \left \| g \right \|_{a_{\text{DG}}\left(\Omega \setminus K_{j,p}\right)}^{2}
&\leq \sum_{K\in \mathcal{T} ^{H}}\int_{K}\left ( \varepsilon \left ( g \right ):C:  \varepsilon \left ( g \right )  \right )\text{dx} +\frac{\gamma }{h}\sum_{E\in \mathcal{E} ^{H}}\int_{E}\underline{{[\![ g]\!]}}:C:\underline{{[\![ g]\!]}}+[\![g]\!]\cdot D\cdot [\![g]\!]\text{ds} \\
&\leq \sum_{K\in \Omega }\int_{K}\varsigma ^{2}\varepsilon \left ( g \right ):C:\varepsilon \left ( g \right )\text{dx}+\frac{\gamma }{h}\sum_{E\in \mathcal{E} ^{H}}\int_{E}\underline{{[\![ g]\!]}}:C:\underline{{[\![\varsigma ^{2} g]\!]}}+[\![g]\!]\cdot D\cdot[\![\varsigma ^{2}g]\!]\text{ds} \\
&=\sum_{K\in \Omega }\int_{K}\varepsilon \left ( g \right ):C: \varepsilon  \left ( \varsigma ^{2} g\right )\text{dx}-2\sum_{K\in \Omega }\int_{K}\varsigma g \varepsilon \left ( \varsigma  \right ):C:  \varepsilon \left ( g \right )\text{dx}\\
&\quad +\frac{\gamma }{h}\sum_{E\in \mathcal{E} ^{H}}\int_{E}\underline{{[\![ \varsigma g]\!]}}:C:\underline{{[\![\varsigma  g]\!]}}+[\![\varsigma g]\!]\cdot D\cdot [\![\varsigma g]\!]\text{ds}.
\end{aligned}
\end{equation}
Again using Lemma \ref{lemmasupp}, there exists$ \ \gamma_{1}\in V$ such that $$\pi\left ( \gamma _{1} \right )=\pi\left ( \varepsilon \left ( \varsigma ^{2}g \right ) \right ), \quad \text{supp}\left ( \gamma _{1} \right )\subset \text{supp}\left ( \pi\left ( \varsigma ^{2}g \right ) \right ).$$
For $K_{s}\subset \Omega \setminus K_{j,p}$, since $\varsigma \equiv 1 $ on $K_{s}$, we note that
$$b_{s}\left ( \varsigma ^{2}g,\psi _{n}^{s} \right )=0,\ \forall \ n=1,2,...,N.$$
On the other hand, for any coarse element $K_{s}\subset K_{j,p-1},$ $\varsigma \equiv 0,$ then we also have
$$b_{s}\left ( \varsigma ^{2}g,\psi _{n}^{s} \right )=0, \ \forall \ n=1,2,...,N.$$
Then by the definition of $g$, we obtain
$$a_{\text{DG}}\left ( g,\varsigma ^{2}g-\gamma _{1} \right )=a_{\text{DG}}\left ( \phi _{i,\text{ms}}^{j},\varsigma ^{2}g-\gamma_{1} \right ).$$
Recall the multiscale basis function $\phi_{i, \text{ms}}^{j}$ satisfies the constraint energy minimization problem\\
	 $a_{\text{DG}}\left ( \phi _{i,\text{ms}}^{j},\varsigma ^{2}g-\gamma _{1}\right )=0,$ therefore, 
\begin{equation}
    \begin{aligned}
 \sum_{K\in \Omega }\int_{K} \varepsilon \left ( g \right ):C: \varepsilon \left ( \varsigma ^{2}g \right )\text{dx}
 &=\sum_{K\in \Omega }\int_{K}\varepsilon \left ( g \right ):C:\varepsilon \left ( \gamma _{1} \right )\text{dx}\\
&\leq M\left \| g \right \|_{a_{\text{DG}}\left ( K_{j,p}\setminus  K_{j,p-1}\right )}\left \| \pi\left ( \varsigma ^{2}g \right ) \right \|_{b\left ( K_{j,p}\setminus K_{j,p-1} \right )}.   
    \end{aligned}
\end{equation}
For $K\subset K_{j,p}\setminus K_{j,p-1},$  $\pi \left(g\right)=0,$ then we have
\begin{equation}
\left \| \pi\left ( \varsigma ^{2}g \right ) \right \|_{b\left ( K \right )}^{2}\leq \left \|  \varsigma ^{2}g \right \|_{b\left ( K \right )}^{2}\leq\frac{1}{c_{1}}\int_{K} \varepsilon \left ( g \right ):C :\varepsilon \left ( g \right )\text{dx},\label{equ:54}
\end{equation}
then there holds
\begin{equation}
    \begin{aligned}
     \left \| \pi\left ( \varsigma ^{2}g \right ) \right \|_{b\left ( K_{j,p}\setminus K_{j,p-1} \right )}\leq c_{1}^{-\frac{1}{2}}\left \| g \right \|_{a_{\text{DG}}\left ( K_{j,p}\setminus K_{j,p-1} \right )}.\label{equ:55}
    \end{aligned}
\end{equation}
Next  $2\sum_{K\in \Omega }\int_{K}\varsigma g \varepsilon \left ( \varsigma  \right ):C:  \varepsilon \left ( g \right )\text{dx}$ can be bounded as
\begin{equation}
    \begin{aligned}
2\sum_{K\in \Omega }\int_{K}\varsigma g \varepsilon \left ( \varsigma  \right ):C:  \varepsilon \left ( g \right )\text{dx} &\leq 2\left \| g \right \|_{b\left ( K_{j,p}\setminus K_{j,p-1} \right )} \left \| g \right \|_{a_{\text{DG}}\left ( K_{j,p}\setminus K_{j,p-1} \right )}\\
& \leq 2c_{1}^{-\frac{1}{2}}\left \| g \right \|^{2}_{a_{\text{DG}}\left ( K_{j,p}\setminus K_{j,p-1} \right )}.\label{equ:56}
    \end{aligned}
\end{equation}
Combining (\ref{equ:54}), (\ref{equ:55}) and (\ref{equ:56}), we arrive at  
\begin{equation}
    \left \| g \right \|_{a_{\text{DG}}\left ( \Omega \setminus K_{j,p} \right )}^{2}\leq \frac{1}{2}\left ( \frac{M}{c_{1}}+2c_{1}^{-\frac{1}{2}}  \right)\left \| g \right \|_{a_{\text{DG}}\left ( K_{j,p} \setminus K_{j,p-1}\right ) }^{2}.
\end{equation}
By the above inequality and the relation of $\Omega \setminus K_{j,p-1}=\left( \Omega \setminus K_{j,p} \right) \cup  \left(K_{j,p} \setminus K_{j,p-1}\right)$, we have
\begin{equation}
    \begin{aligned}
 \left \| g \right \|_{a_{\text{DG}}\left ( \Omega \setminus K_{j,p-1} \right )}^{2}&=\left \| g \right \|_{a_{\text{DG}}\left ( \Omega \setminus K_{j,p} \right )}^{2}+\left \| g \right \|_{a_{\text{DG}}\left ( K_{j,p} \setminus K_{j,p-1} \right )}^{2}\\
&\geq     \left [ 2c_{1}^{-\frac{1}{2}}\left ( Mc_{1}^{-\frac{1}{2}}+2 \right )^{-1}+1 \right ]  \left \| g \right \|_{a_{\text{DG}}\left (\Omega \setminus K_{j,p}\right )}^{2}  \\
&\geq \left [ 1+\left ( \frac{M}{2}+c_{1}^{\frac{1}{2}} \right )^{-1} \right ]  \left \| g \right \|_{a_{\text{DG}}\left (\Omega \setminus K_{j,p}\right )}^{2}.
    \end{aligned}
\end{equation}
$\mathbf{Part3.}$ Using (\ref{part2}) in (\ref{part1}), we note that
\begin{equation}
    \left \| \phi _{i,\text{ms}}^{j}- \phi _{i,\text{glo}}^{j} \right \|_{\text{DG}} 
    \leq 2\sqrt{M} \left ( c_{1}+\frac{1}{c_{1}} \right )  \left [ 1+\left ( \frac{M}{2}+c_{1}^{\frac{1}{2}} \right )^{-1} \right ] ^{-\frac{1}{2}}\left \| g \right \|_{a_{\text{DG}}\left ( \Omega \setminus K_{j,p-1} \right )}\label{part3}.
\end{equation}
By using (\ref{part2}) in (\ref{part3}), we obtain
\begin{equation}
\begin{aligned}
 \left \| \phi _{i,\text{ms}}^{j}- \phi _{i,\text{glo}}^{j} \right \|_{\text{DG}} 
 &\leq  2\sqrt{M} \left ( c_{1}+\frac{1}{c_{1}} \right )  \left [ 1+\left ( \frac{M}{2}+c_{1}^{\frac{1}{2}} \right )^{-1} \right ] ^{\frac{1}{2}-p}\left \| g \right \|_{a_{\text{DG}}\left ( \Omega \setminus K_{j} \right )}^{2}\\
 &\leq 2\sqrt{M} \left ( c_{1}+\frac{1}{c_{1}} \right )  \left [ 1+\left ( \frac{M}{2}+c_{1}^{\frac{1}{2}} \right )^{-1} \right ] ^{\frac{1}{2}-p}\left \| g \right \|_{\text{DG}}^{2}.
\end{aligned}
\end{equation}
By the definition of $g$, we have
$$ \left \| g \right \|_{\text{DG}}= \left \| \phi _{i,\text{ms}}^{j}-\tilde{\phi }_{i}^{j} \right \|_{\text{DG}}  \leq  2\left \| \tilde{\phi }_{i}^{j} \right \|_{\text{DG}}\leq 2M^{\frac{1}{2}}\left \| \psi  _{i}^{j} \right \|_{b\left ( K_{j} \right )}.$$
 
 This completes the proof.
\end{proof}

We will next prove a stable decomposition property. In particular, we will now estimate $ \left \|\sum_{j=1}^{N}\left ( \phi _{i,\text{ms}}^{j}- \phi _{i,\text{glo}}^{j} \right )  \right \|_{\text{DG}}$  based on Lemma \autoref{lemma:dgcoupling}, that is, the difference between global multiscale basis functions $\phi _{i,\text{glo}}^{j}$  and localized multiscale basis functions $\phi _{i,\text{ms}}^{j}$ on whole domain.

\begin{lemma}\label{lemma4.4}
With the same notations in Lemma \ref{lemma:dgcoupling}, then there exists a constant $c_3>0$ such that 
\begin{equation}
    \left \|\sum_{j=1}^{N}\left ( \phi _{i,\text{ms}}^{j}- \phi _{i,\text{glo}}^{j} \right )  \right \|_{\text{DG}} ^{2}\leq c_{3}\left ( p+1 \right )^{d}\sum_{j=1}^{N}\left \|\left ( \phi _{i,\text{ms}}^{j}- \phi _{i,\text{glo}}^{j} \right )  \right \|_{\text{DG}} ^{2}.
\end{equation}
\end{lemma}
\begin{proof}
 Let $\vartheta =\sum_{j=1}^{N}\left ( \phi _{i,\text{ms}}^{j}- \phi _{i,\text{glo}}^{j} \right ).$ By the Lemma\ \ref{lemmasupp}, $\exists \ y_{j}\in V_h,$ such that
$$ \pi\left ( y_{j} \right )=\pi \left ( \left ( 1-\chi _{j}^{p+1,p} \right )\vartheta  \right ),\quad \text{supp}\left ( y_{j} \right )\subset K_{i,p+1}\setminus K_{j,p},\quad \left \| y_{j} \right \|_{\text{DG}}\leq M\left \| \pi\left ( \left ( 1-\chi _{i}^{p+1,p} \right )\vartheta  \right ) \right \|_{b}.  $$
 Then we obtian 
$$ a_{\text{DG}}\left ( \phi _{i,\text{ms}}^{j}- \phi _{i,\text{glo}}^{j}, v \right )+b\left ( v, \left ( {\psi }' \right ) _{i}^{j}-\left ( {\psi }' \right ) _{i,p}^{j}\right )=0,\ \forall \ v\in V_{h}\left ( K_{j,p} \right ).$$
 Let $v=\left ( \left ( 1-\chi _{j}^{p+1,p} \right )\vartheta  \right )-y_{j}$, then
 \begin{equation}
   a_{\text{DG}}\left ( \phi _{i,\text{ms}}^{j}- \phi _{i,\text{glo}}^{j}, \left ( \left ( 1-\chi _{j}^{p+1,p} \right )\vartheta  \right )-y_{j} \right )=0  .
 \end{equation}
Hence, by the definition of $\vartheta$, we note that
\begin{equation}
\begin{aligned}
 \left \| \vartheta  \right \|_{\text{DG}}^{2}
   &=a_{\text{DG}}\left ( \vartheta , \vartheta  \right )\\
   &=\sum_{j=1}^{N}a_{\text{DG}}\left ( \phi _{i,\text{ms}}^{j}- \phi _{i,\text{glo}}^{j}, \vartheta  \right )\\
   &=\sum_{j=1}^{N}a_{\text{DG}}\left ( \phi _{i,\text{ms}}^{j}- \phi _{i,\text{glo}}^{j}, \chi _{j}^{p+1,p}\vartheta +y_{j}  \right ). 
\end{aligned}
\end{equation}

For each $j=1,2,...,N,$ there exists a constant $c_3>0$ such that
\begin{equation}
\begin{aligned}
 \left \| \chi _{j}^{p+1,p}\vartheta  \right \|_{\text{DG}}^{2} 
 &\leq c_{3}\left ( \left \| \vartheta  \right \|_{a_{\text{DG}}\left ( K_{j,p+1} \right )}^{2}+  \left \| \vartheta  \right \|_{b\left ( K_{j,p+1} \right )}^{2}\right )\\
&\leq c_{3} \left ( 1+\frac{1}{c_{3}} \right )
^{2}\left \| \vartheta  \right \|_{a_{\text{DG}}\left(K_{j,p+1} \right )}^{2}.
\end{aligned}
\end{equation}
In addition, by the projection of $\pi$ and the definition of the cutoff function, we have
\begin{equation}
\begin{aligned}
 \left \| y_{j} \right \|_{\text{DG}}^{2}
 &\leq M^{2}\left \| \pi\left ( 1-\chi _{j}^{p+1,p} \right )\vartheta  \right \|_{b}^{2}\\
 &\leq M^{2}\left \| \pi\left (  \chi _{j}^{p+1,p}\vartheta  \right )\right \|_{b\left ( K_{j,p+1} \right )}^{2}\\
 &\leq M^{2}\left \| \vartheta  \right \|_{b\left ( K_{j,p+1} \right )}^{2}.
\end{aligned}
\end{equation}
This completes the proof of the lemma.
\end{proof}

Now, we are ready to establish our main theorem, which estimates the error between the solution $u$ and the multiscale
solution $u_{\text{ms}}^{\text{off}}$. To proof the following theorem, the approximate solution $u_{\text{gl}} \in V_{\text{cem}}$ obtained in the global multiscale space $V_{\text{cem}}$ is defined by
\begin{equation}
a_{\text{DG}}\left(u_{\text{gl}}, v\right)=\int_{\Omega} f v,\ \forall\  v \in V_{\text{cem}}.\label{equ:global}
\end{equation}

\begin{theorem}\label{thm:thm1}
Let $u$ be the solution of (\ref{adg}) and $u_{\text{ms}}^{\text{off}}$ be the solution of (\ref{eq:ms}). Then we have
\begin{equation}
    \left \| u-u_{\text{ms}}^{\text{off}} \right \|_{\text{DG}}\leq c_{3}c_{1}^{-\frac{1}{2}}\left \| k_{1}^{-\frac{1}{2}}f \right \|+c_{3}\left ( p+1 \right )^{\frac{d}{2}}\sqrt{M_{1}}\left \| u_{\text{gl}} \right \|_b, 
\end{equation}
where $u_{\text{gl}}$ is the multiscale solution using corresponding global basis in (\ref{equ:global}). \\
Moreover, if $p=O\left(\log \left(\frac{\lambda _{\max}\left ( C\left ( x \right ) \right )}{H}\right)\right)$ and $\chi_{i}$ are bilinear partition of unity, then there exists a constant $c_4$ such that 
\begin{equation}
\left \| u-u_{\text{ms}}^{\text{off}} \right \|_{\text{DG}} \leq c_{4}c_{1}^{-\frac{1}{2}}H\left \| k_{2}^{-\frac{1}{2} }f \right \| .
\end{equation}
\end{theorem}

\begin{proof}
\quad Let $\left ( \chi _{j} \right )_{j=1}^{N}$ be bilinear partition of unity. 
Define $u_{\text{gl}}=\sum_{j=1}^{N}\sum_{i=1}^{G_{j}}d_{ij}\phi _{i,\text{glo}}^{j}$, then
$v=\sum_{j=1}^{N}\sum_{i=1}^{G_{j}}d_{ij} \phi _{i,\text{ms}}^{j}\in V_{H}$.
So, we note that
\begin{equation}
    \begin{aligned}
    \left \| u-u_{\text{ms}}^{\text{off}} \right \|_{\text{DG}}&\leq \left \| u-v \right \|_{\text{DG}} \\
    &\leq  \left \| u-u_{\text{gl}} \right \|_{\text{DG}}+\left \| \sum_{j=1}^{N}\sum_{i=1}^{G_{j}}d_{ij}\left ( \phi _{i,\text{ms}}^{j} - \phi _{i,\text{glo}}^{j}\right ) \right \|_{\text{DG}}.\label{equ:69}
    \end{aligned}
\end{equation} 
For the second term in the right hand of (\ref{equ:69}), recall the Lemma \ref{lemmasupp} and Lemma \ref{lemma:dgcoupling}, there exists a constant $c_3>0$ such that 
\begin{equation}
    \begin{aligned}
\left \| \sum_{j=1}^{N}\sum_{i=1}^{G_{j}}d_{ij}\left ( \phi _{i,\text{ms}}^{j} - \phi _{i,\text{glo}}^{j}\right ) \right \|_{\text{DG}}^{2}
&\leq c_{3}\left ( p+1 \right )^{d}\sum_{j=1}^{N}\left \|\sum_{i=1}^{G_{j}}d_{ij}\left ( \phi _{i,\text{ms}}^{j} - \phi _{i,\text{glo}}^{j}\right )  \right \|_{\text{DG}}^{2}\\
&\leq c_{3}\left ( p+1 \right )^{d}M_{1}\left \|\sum_{i=1}^{G_{j}}d_{ij}\psi _{i}^{j} \right \|_{b}^{2}\\
&\leq c_{3}\left ( p+1 \right )^{d}M_{1}\left \| u_{\text{gl}} \right \|_{b}^{2}.
    \end{aligned}
\end{equation}
Next we estimate the first term $\left \| u-u_{\text{gl}} \right \|_{\text{DG}}$ in (\ref{equ:69}). By using the orthogonality property and the eigenvalues in (\ref{abf}), we have
\begin{equation}
\begin{aligned}
 a_{\text{DG}}\left(u-u_{\text{gl}}, u-u_{\text{gl}}\right) &=a_{\text{DG}}\left(u, u-u_{\text{gl}}\right)=\left(f, u-u_{\text{gl}}\right) \\ 
 & \leq\left\|k_{1} ^{-\frac{1}{2}}f\right\|_{L^{2}(\Omega)}\left\|u-u_{\text{gl}}\right\|_{b}
\end{aligned}
\end{equation}
and
\begin{equation}
\begin{aligned}
\left \| u-u_{\text{gl}} \right \|_{b}^{2}=\sum_{i=1}^{N}\left\|\left(u-u_{\text{gl}}\right)\right\|_{b\left(K_{i}\right)}^{2}
\leq \frac{1}{c_{1}} \sum_{i=1}^{N}\left\|u-u_{\text{gl}}\right\|_{a_{\text{DG}}\left(K_{i}\right)}^{2}=\frac{1}{c_{1}}\left\|u-u_{\text{gl}}\right\|_{\rm{DG}}^{2}.
\end{aligned}
\end{equation}
Then we apply Lemma \ref{lemmasupp} to the function 
$\sum_{i=1}^{G_{j}}d_{ij}\left ( \phi _{i,\text{ms}}^{j} - \phi _{i,\text{glo}}^{j}\right )$. 
Then we have
\begin{equation}
\left \| u-u_{\text{ms}}^{\text{off}} \right \|_{\text{DG}}\leq c_{3}c_{1}^{-\frac{1}{2}}\left \| k_{1}^{-\frac{1}{2}}f \right \|+c_{3}\left ( p+1 \right )^{\frac{d}{2}}\sqrt{M_{1}}\left \| u_{\text{gl}} \right \|_b.\label{equ:73}
\end{equation}
This completes the first part of the theorem.  Next,  we estimate $\left \| u_{\text{gl}}\right \|_{b}$,
\begin{equation*}
\begin{aligned}
\left \|  u_{\text{gl}}\right \|_{b}^{2}
&\leq \max\left \{ k_{1}  \right \}   \left \| u_{\text{gl}} \right \|_{L^{2}\left ( \Omega  \right )}^{2}\\
&\leq c_{3}\max \left \{ k_{1}  \right \}       \left \| u_{\text{gl}}\right \|_{\text{DG}}^{2}\\
&\leq c_{3}\max\left \{ k_{1}  \right \}   \int_{\Omega }f u_{\text{gl}}\\
&\leq c_{3}\max\left \{ k_{1}  \right \}   \left \| k_{1}^{-\frac{1}{2}}f \right \|_{L^{2}\left ( \Omega  \right )}\left \| u_{\text{gl}} \right \|_{b}.
\end{aligned}
\end{equation*}
Dividing both sides of the above inequality by $\left \| u_{\text{gl}} \right \|_{b}$, we obtian
\begin{equation}
\left \| u_{\text{gl}}\right \|_{b} \leq c_{3}\max\left \{ k_{1}  \right \}   \left \| k_{1}^{-\frac{1}{2}}f \right \|_{L^{2}\left ( \Omega\right )}.\label{equ:111}
\end{equation}
Substituting (\ref{equ:111}) into (\ref{equ:73}), we have 
\begin{equation}
\left \| u-u_{\text{ms}}^{\text{off}} \right \|_{\text{DG}}\leq c_{1}^{-\frac{1}{2}}c_{3}\left(1+c_{1}^{\frac{1}{2}}\max\{k_1\} c_{3}\left ( p+1 \right )^{\frac{d}{2}}\sqrt{M_{1}}\right)\left \| k_{1}^{-\frac{1}{2}}f \right \|. \label{equ:222}
\end{equation}

We further demonstrate that our method achieves linear convergence in coarse grid size.  To obtain our desired goal,
using the fact that $\left|\nabla \chi_i\right|=O\left(H^{-1}\right)$ and $k_1=\sum_{i=1}^{N_{c}}k_2\left|\nabla \chi_{i}\right |^{2},$  and assuming $\chi_{i}$ are bilinear partition of unity,  (\ref{equ:222}) can be rewritten as
 \begin{equation}
 \left \| u-u_{\text{ms}}^{\text{off}} \right \|_{\text{DG}}\leq c_{1}^{-\frac{1}{2}}c_{3}\left(1+c_{1}^{\frac{1}{2}}\max\{k_1\} c_{3}\left ( p+1 \right )^{\frac{d}{2}}\sqrt{M_{1}}\right)H\left \| k_{2}^{-\frac{1}{2}}f \right \|.
 \end{equation}
Finally we only need to proof  $c_{1}^{\frac{1}{2}}\max\{k_1\} c_{3}\left ( p+1 \right )^{\frac{d}{2}}\sqrt{M_{1}}\sim O\left ( 1 \right ).$ Taking logarithm to $c_{1}^{\frac{1}{2}}\max\left \{ k_{1}  \right \}  c_{3}\left ( p+1 \right )^{\frac{d}{2}}\sqrt{M_{1}}$, we have 
 $$\text{log}\left (\max k_2 \right )+\frac{d}{2}\text{log}\left ( p+1 \right )+\frac{1}{2}\left ( \frac{1}{2}-p \right )\text{log}\left [ 1+\left ( \frac{M}{2}+c_{1}^{\frac{1}{2}} \right )^{-1} \right ]+\text{log}\left ( H^{-2} \right )=O\left ( 1 \right ).$$
 Let $p=O\left(\log \left(\frac{\lambda _{\max}\left ( C\left ( x \right ) \right )}{H}\right)\right),$  there holds 
 $$\max\left \{ k_2 \right \}  H^{-2}c_3 \left ( p+1 \right )^{\frac{d}{2}}M_{1}^{\frac{1}{2}}=O\left ( 1 \right ).$$
Then we take $c_4=c_{3}\left(1+c_{1}^{\frac{1}{2}}\max\{k_1\} c_{3}\left ( p+1 \right )^{\frac{d}{2}}\sqrt{M_{1}}\right).$ This completes the proof.
\end{proof}

\subsection{Convergence Analysis of Relaxed Method} 
In this section, we analyze the convergence of relaxed version of the CEM-GMsFEM. We first consider the basis functions defined in Section \ref{sec:relax} and estimate the difference between $ \phi _{i,\text{ms}}^{j}$ and $  \phi _{i,\text{glo}}^{j}$.
\begin{lemma}\label{lemma4.5}
Let us consider such an oversampling area $K_{j,p}$. That is, $K_{j,p}$ is obtained by mapping and extending the oversampling area $K_{j}$ with $p$ coarse grid blocks. Let $ \psi _{j}^{i}\in W_{\text{aux}}$ be the auxiliary multiscale basis function, $ \phi _{i,\text{glo}}^{j} $ be the solution of   (\ref{basisrelexiob}), and $\phi_{i,\text{ms}}^{j} $ be the solution of   (\ref{basisrelexip}). Then we have
\begin{equation}
      \left \| \phi _{i,\text{ms}}^{j}- \phi _{i,\text{glo}}^{j} \right \|_{\text{DG}}^{2}+\left \| \pi \left(\phi _{i,\text{ms}}^{j}- \phi _{i,\text{glo}}^{j}\right) \right \|_{b}^{2}\leq M_{2}\left(\left \| \psi _{i,\text{glo}}^{j} \right \|_{\text{DG}}^{2}+\left \| \pi(\left(\psi _{i,\text{glo}}^{j}\right) \right \|_{b}^{2}\right),
 \end{equation}
 where $M_{2}=3\left ( 1+\frac{1}{2}c_{1}^{-1} \right )\left ( 1+ \left ( 2+\frac{1}{2}c_{1}^{-\frac{1}{2}} \right )^{-1} \right )^{1-p}.$
\end{lemma}
\begin{proof}
By the above definition, we note that
\begin{equation*}
    \begin{aligned}
    a_{\text{DG}}\left (  \phi _{i,\text{glo}}^{j},w \right )+b\left ( \pi\left (  \phi _{i,\text{glo}}^{j} \right ) ,\pi\left ( w \right )\right )&=b\left (  \phi _{i,\text{glo}}^{j}, \pi\left ( w \right )\right ),\quad  \forall \ w\in  V_h,\\
    a_{\text{DG}}\left ( \phi _{i,\text{ms}}^{j}, w \right )+b\left ( \pi\left ( \phi _{i,\text{ms}}^{j} \right ) ,\pi\left ( w \right )\right )&=b\left ( \phi _{i,\text{ms}}^{j}, \pi\left ( w \right )\right ),\quad  \forall \ w\in  V_{h}\left ( K_{j,p} \right ).
    \end{aligned}
\end{equation*}
Subtracting the above equations, we obtain
\begin{equation*}
   \begin{aligned}
   a_{\text{DG}}\left (  \phi _{i,\text{glo}}^{j},w \right )-a_{\text{DG}}\left ( \phi _{i,\text{ms}}^{j}, w \right )+b\left ( \pi\left (  \phi _{i,\text{glo}}^{j} \right ) ,\pi\left ( w \right )\right )-b\left ( \pi\left ( \phi _{i,\text{ms}}^{j} \right ) ,\pi\left ( w \right )\right )\\=b\left (  \phi _{i,\text{glo}}^{j}, \pi\left ( w \right )\right )-b\left ( \phi _{i,\text{ms}}^{j}, \pi\left ( w \right )\right ).
   \end{aligned} 
\end{equation*}
Then we have
\begin{equation*}
   \begin{aligned}
a_{\text{DG}}\left (  \phi _{i,\text{glo}}^{j}- \phi _{i,\text{ms}}^{j}, w\right )+b\left ( \pi\left (  \phi _{i,\text{glo}}^{j} \right ) -\pi\left ( \phi _{i,\text{ms}}^{j} \right ),\pi\left ( w \right )\right )=0,
   \end{aligned} 
\end{equation*}
for all $w \in V_{h}\left(K_{j,p}\right) $.
Taking $w=v- \phi _{i,\text{glo}}^{j}, w\in V_{h}\left ( K_{j,p} \right )$, where $v\ \in\ V_{h}\left(K_{j,p}\right)$, we note that
$$\left \| \phi _{i,\text{ms}}^{j}-  \phi _{i,\text{glo}}^{j}\right \|_{\text{DG}}^{2}+
\left \| \pi\left ( \phi _{i,\text{glo}}^{j}-  \phi _{i,\text{glo}}^{j} \right )\right \|_{b}^{2}\leq \left \| \phi _{i,\text{glo}}^{j}- v\right \|_{\text{DG}}^{2}+\left \| \pi\left ( \phi _{i,\text{glo}}^{j}- v \right )\right \|_{b}^{2}.$$
Let $v=\chi _{j}^{p,p-1}\phi _{i,\text{glo}}^{j},$ we have
$$\left \| \phi _{i,\text{ms}}^{j}-  \phi _{i,\text{glo}}^{j}\right \|_{\text{DG}}^{2}+
\left \| \pi\left ( \phi _{i,\text{ms}}^{j}-  \phi _{i,\text{glo}}^{j} \right )\right \|_{\text{DG}}^{2}\leq \left \| \phi _{i,\text{glo}}^{j}- \chi _{j}^{p,p-1}\phi _{i,\text{glo}}^{j}\right \|_{\text{DG}}^{2}+\left \| \pi\left ( \phi _{i,\text{glo}}^{j}-  \chi _{j}^{p,p-1}\phi _{i,\text{glo}}^{j} \right )\right \|_{b}^{2}.$$
Subsequently, we will estimate the above inequality in three parts, first estimating $\left \| \phi _{i,\text{glo}}^{j}- \chi _{j}^{p,p-1}\phi _{i,\text{glo}}^{j}\right \|_{\text{DG}}^{2}$ and $\left \| \pi\left ( \phi _{i,\text{glo}}^{j}-  \chi _{j}^{p,p-1}\phi _{i,\text{glo}}^{j} \right )\right \|_{b}^{2}$ separately, then further estimating the first part's results in the second part, and finally obtaining our final estimate of the second part in the third part.\\

$\mathbf{Part\ 1.}$\ Next we estimate $\left \| \left ( 1-\chi _{j}^{p+1,p} \right )\phi _{i,\text{glo}}^{j} \right \|_{\text{DG}}$. By the definition of \rm{DG} norm and combining the two occurrences of $\left ( 1-\chi _{j}^{p+1,p} \right )$ in the first term but in different positions as $\left ( 1-\chi _{j}^{p+1,p} \right )^2$and the two occurrences of $\phi _{i,\text{glo}}^{j}$ in the first term but in different positions as $\left(\phi _{i,\text{glo}}^{j}\right)^2$,  we can get
\begin{equation}
\begin{aligned}
\left \| \left ( 1-\chi _{j}^{p+1,p} \right )\phi _{i,\text{glo}}^{j} \right \|_{\text{DG}}^{2}
&=\sum_{K\in\mathcal{T} ^{H}}\int_{K} \epsilon \left ( \left ( 1-\chi _{j}^{p+1,p} \right )\phi _{i,\text{glo}}^{j}   \right ):C:\epsilon \left (\left ( 1-\chi _{j}^{p+1,p} \right )\phi _{i,\text{glo}}^{j}  \right ) \text{dx}\\
&\quad+\frac{\gamma }{h}\sum_{E\in \mathcal{E}^{H}}\int_{E}\underline{{[\![ \left ( 1-\chi _{j}^{p+1,p} \right )\phi _{i,\text{glo}}^{j} ]\!] }}:C:\underline{{[\![ \left ( 1-\chi _{j}^{p+1,p} \right )\phi _{i,\text{glo}}^{j} ]\!] } }\\
&\quad\quad+[\![ \left ( 1-\chi _{j}^{p+1,p} \right )\phi _{i,\text{glo}}^{j} ]\!] \cdot  D \cdot [\![ \left ( 1-\chi _{j}^{p+1,p} \right )\phi _{i,\text{glo}}^{j} ]\!]\text{ds}\\
&\leq 2\left [ \sum_{\Omega \setminus K_{j,p}\in\mathcal{T}  ^{H}}\int_{\Omega \setminus K_{j,p}}\left ( 1-\chi _{j}^{p+1,p} \right )^{2}\varepsilon \left ( \phi _{i,\text{glo}}^{j}   \right ):C:\varepsilon \left ( \phi _{i,\text{glo}}^{j}   \right )\text{dx}\right]\\
&\quad+2\left[\sum_{\Omega \setminus K_{j,p}\in\mathcal{T} ^{H}}\int_{\Omega \setminus K_{j,p}} \varepsilon \left ( \chi _{j}^{p+1,p} \right ):C: \varepsilon \left ( \chi _{j}^{p+1,p} \right )\left ( \phi _{i,\text{glo}}^{j}  \right )^{2} \text{dx}\right ]\\
&\quad+\frac{\gamma }{h}\sum_{E\in \mathcal{E}^{H}}\int_{E}\underline{{[\![ \left ( 1-\chi _{j}^{p+1,p} \right )\phi _{i,\text{glo}}^{j} ]\!] }}:C:\underline{{[\![ \left ( 1-\chi _{j}^{p+1,p} \right )\phi _{i,\text{glo}}^{j} ]\!] } }\text{ds}\\
&\quad+\frac{\gamma }{h}\sum_{E\in \mathcal{E}^{H}}\int_{E}[\![ \left ( 1-\chi _{j}^{p+1,p} \right )\phi _{i,\text{glo}}^{j} ]\!] \cdot  D \cdot [\![ \left ( 1-\chi _{j}^{p+1,p} \right )\phi _{i,\text{glo}}^{j} ]\!]\text{ds}.\\
\end{aligned}
\end{equation}
By the definition of cutoff function, note that $1-\chi _{j}^{p+1,p}\leq 1,$ then the above inequality can be calculated as
\begin{equation}
\begin{aligned}
 \left \| \left ( 1-\chi _{j}^{p+1,p} \right )\phi _{i,\text{glo}}^{j}  \right \|_{\text{DG}}^{2}
 \leq 2\left ( \left \| \phi _{i,\text{glo}}^{j}   \right \|_{a_{\text{DG}}\left ( \Omega \setminus K_{j,p-1} \right )}^{2}+c_{1}\left \| \phi _{i,\text{glo}}^{j}   \right \|_{b\left ( \Omega \setminus K_{j,p-1} \right )}^{2} \right ).
\end{aligned}
\end{equation}
Note that for $K\in \mathcal{T}^{H}$, we pick the reflection $P:=I-\pi$
which satisfies 
\begin{equation}
    \left \| Pw \right \|_{b\left ( K  \right )}^{2}\leq \frac{1}{2}c_{1}^{-1}a_{\text{DG}}\left ( w,w \right ),\label{39}
\end{equation}
then we have
\begin{equation*}
    \begin{aligned}
     \left \| \phi _{i,\text{glo}}^{j}  \right \|_{b\left ( K \right )}^{2}&=\left \| P\left ( \phi _{i,\text{glo}}^{j}  \right )+\pi\left ( \phi _{i,\text{glo}}^{j} \right ) \right \|_{b\left ( K \right )}^{2}\\
&=\left \| P\left ( \phi _{i,\text{glo}}^{j}  \right ) \right \|_{b\left ( K \right )}^{2}+\left \| \pi\left ( \phi _{i,\text{glo}}^{j}  \right ) \right \|_{b\left ( K \right )}^{2}\\
&\leq \frac{1}{2}c_{2}^{-1}\left \| \phi _{i,\text{glo}}^{j}  \right \|_{a_{\text{DG}}\left ( K \right )}^{2}+\left \| \pi\left ( \phi _{i,\text{glo}}^{j}  \right ) \right \|_{b\left ( K \right )}^{2}.
    \end{aligned}
\end{equation*}
Therefore, we have
\begin{equation}
    \begin{aligned}
\left \| \left ( 1-\chi _{j}^{p+1,p} \right )\phi _{i,\text{glo}}^{j}  \right \|_{\text{DG}}^{2}\leq 
 2 \left [ \frac{3}{2}\left \| \phi _{i,\text{glo}}^{j}   \right \|_{a_{\text{DG}}\left ( \Omega \setminus K_{j,p-1} \right )}^{2}+\left \| \pi\left ( \phi _{i,\text{glo}}^{j}   \right ) \right \|_{b\left ( \Omega \setminus K_{j,p-1} \right )}^{2} \right ].\label{equ:74}
    \end{aligned}
\end{equation}

We will estimate the second term $\left\| \pi \left ( \phi _{i,\text{glo}}^{j}  - \chi _{j}^{p,p-1}  \phi _{i,\text{glo}}^{j} \right) \right\|_{b}^{2} $. Again by the definition of cutoff function, note that $\chi _{j}^{p+1,p}\leq 1,$ we have
\begin{equation}
\begin{aligned}
 \left\| \pi \left ( \phi _{i,\text{glo}}^{j}  - \chi _{j}^{p,p-1} \phi _{i,\text{glo}}^{j}   \right)\right\|_{b}^{2}&=\left\| \pi \left(\left ( 1- \chi _{j}^{p,p-1}  \right)\phi _{i,\text{glo}}^{j} \right) \right\|_{b}^{2}\\
 &\leq \left\| \left ( 1- \chi _{j}^{p,p-1}  \right)\phi _{i,\text{glo}}^{j} \right\|_{b}^{2}\\
 &\leq \left \| \phi _{i,\text{glo}}^{j}  \right \|_{b\left ( \Omega \setminus K_{j,p-1} \right )}^{2}\\
 &\leq \frac{1}{2} c_{1}^{-1}\left \| \phi _{i,\text{glo}}^{j}  \right \|_{a
_{\text{DG}}\left ( \Omega \setminus K_{j,p-1} \right )}^{2}+\left \| \pi\left ( \phi _{i,\text{glo}}^{j}  \right ) \right \|_{b\left ( \Omega \setminus K_{j,p-1} \right )}^{2}.\label{equ:75}
\end{aligned}
\end{equation}
Hence, combining (\ref{equ:74}) and (\ref{equ:75}), we can get
\begin{equation*}
    \begin{aligned}
  &  \quad  \left \| \phi _{i,\text{ms}}^{j}- \phi _{i,\text{glo}}^{j} \right \|_{\text{DG}}^{2}+\left \| \pi \left(\phi _{i,\text{ms}}^{j}- \phi _{i,\text{glo}}^{j}\right) \right \|_{b}^{2}\\
       &  \leq \left ( 3+\frac{1}{2}c_{1}^{-1} \right )\left \| \phi _{i,\text{glo}}^{j}  \right \|_{a_{\text{DG}}\left ( \Omega \setminus K_{j,p-1} \right )}^{2}+3\left \| \pi\left ( \phi _{i,\text{glo}}^{j}  \right ) \right \|_{b\left ( \Omega \setminus K_{j,p-1} \right )}^{2}\\
       &\leq 3\left ( 1+\frac{1}{2}c_{1}^{-1} \right )\left(\left \| \phi _{i,\text{glo}}^{j}  \right \|_{a_{\text{DG}}\left ( \Omega \setminus K_{j,p-1} \right )}^{2}+\left \| \pi\left ( \phi _{i,\text{glo}}^{j}  \right ) \right \|_{b\left ( \Omega \setminus K_{j,p-1} \right )}^{2}\right).
    \end{aligned}
\end{equation*}
$\mathbf{Part\ 2.}$\ We will estimate $\left \| \phi _{i,\text{glo}}^{j}  \right \|_{a_{\text{DG}}\left ( \Omega \setminus K_{j,p-1} \right )}^{2}$ and $\left \| \pi\left ( \phi _{i,\text{glo}}^{j}  \right ) \right \|_{b\left ( \Omega \setminus K_{j,p-1} \right )}^{2}$. Actually, this can be estimated by $ \left \| \phi _{i,\text{glo}}^{j}  \right \|_{a_{\text{DG}}\left ( K_{j,p-1} \setminus K_{j,p-2} \right )}^{2}$ and $\left \| \pi\left ( \phi _{i,\text{glo}}^{j}  \right ) \right \|_{b\left ( K_{j,p-1} \setminus K_{j,p-2} \right )}^{2}$. Then we note that
\begin{equation}
    \begin{aligned}
      a_{\text{DG}}\left ( \phi _{i,\text{glo}}^{j} , \left ( 1-\chi _{j}^{p-1,p-2} \right )\phi _{i,\text{glo}}^{j}  \right )+b\left ( \pi\left ( \phi _{i,\text{glo}}^{j}  \right ), \pi\left ( \left ( 1-\chi _{j}^{p-1,p-2} \right )\phi _{i,\text{glo}}^{j}  \right ) \right )\\
=b\left ( \phi _{i,\text{glo}}^{j} , \pi\left ( \left ( 1-\chi _{j}^{p-1,p-2} \right )\phi _{i,\text{glo}}^{j}  \right ) \right )=0,
    \end{aligned}
\end{equation}
by using the fact that 
$$\text{supp}\left ( 1-\chi _{j}^{p-1,p-2} \right )\subset \Omega \setminus K_{j,p-2},\quad \text{supp}\left ( \phi  _{i,\text{glo}}^{j} \right )\subset K_{j}.$$
By the definition of DG norm, we have
\begin{equation*}
    \begin{aligned}
a_{\text{DG}}\left ( \phi _{i,\text{glo}}^{j} , \left ( 1-\chi _{j}^{p-1,p-2} \right )\phi _{i,\text{glo}}^{j}  \right )
&=\sum_{K_{j,p-2}\in \mathcal{T} ^{H}}\int_{\Omega \setminus K_{j,p-2}}  \epsilon \left (\phi _{i,\text{glo}}^{j}  \right ):C: \epsilon \left (\left ( 1-\chi _{j}^{p-1,p-2} \right )\phi _{i,\text{glo}}^{j}  \right )\text{dx}\\
&\quad+\frac{\gamma }{h}\sum_{E\in \mathcal{E}^{H} }\int_{E}\underline{ {{[\![\phi _{i,\text{glo}}^{j} ]\!]}}}:C:\underline{{{[\![\left ( 1-\chi _{j}^{p-1,p-2} \right )\phi _{i,\text{glo}}^{j} ]\!]}}} \text{ds}\\
&\quad +\frac{\gamma }{h}\sum_{E\in \mathcal{E} ^{H}}\int_{E}{[\![\left ( 1-\chi _{j}^{p-1,p-2} \right )\phi _{i,\text{glo}}^{j} ]\!]}\cdot D\cdot {[\![\phi _{i,\text{glo}}^{j}  ]\!]}\text{ds}\\
&\leq \sum_{K_{j,p-2}\in \mathcal{T} ^{H}}\int_{\Omega \setminus K_{j,p-2}}  \left ( 1-\chi _{j}^{p-1,p-2} \right )\epsilon \left (\phi _{i,\text{glo}}^{j}  \right ):C: \epsilon \left (\phi _{i,\text{glo}}^{j}  \right )\text{dx}\\
&\quad-\sum_{K_{j,p-2}\in \mathcal{T} ^{H}}\int_{\Omega \setminus K_{j,p-2}}  \phi _{i,\text{glo}}^{j} \varepsilon  \left ( \chi _{j}^{p-1,p-2} \right ): C: \epsilon \left (\phi _{i,\text{glo}}^{j}  \right )\text{dx}\\
&\quad+\frac{\gamma }{h}\sum_{E\in \mathcal{E}^{H}}\int_{E}\underline{ {{[\![\phi _{i,\text{glo}}^{j} ]\!]}}}:C:\underline{{{[\![\phi _{i,\text{glo}}^{j} ]\!]}}} \text{ds}\\
&\quad +\frac{\gamma }{h}\sum_{E\in \mathcal{E}^{H}}\int_{E}{[\![\phi _{i,\text{glo}}^{j} ]\!]}\cdot D\cdot {[\![\phi _{i,\text{glo}}^{j}  ]\!]}\text{ds}.\\
    \end{aligned}
\end{equation*}

Consequently, we will estimate  the global multiscale basis functions $\phi _{i,\text{glo}}^{j} $ on domain $\Omega \setminus K_{j,p-1}$ and have
\begin{equation}
    \begin{aligned}
&\quad    \left \| \phi _{i,\text{glo}}^{j}  \right \|_{a_{\text{DG}}\left ( \Omega \setminus K_{j,p-1} \right )}^{2}\\
&\leq a_{\text{DG}}\left ( \phi _{i,\text{glo}}^{j} , \left ( 1-\chi _{j}^{p-1,p-2} \right )\phi _{i,\text{glo}}^{j}  \right )
\quad +\left \| \phi _{i,\text{glo}}^{j}  \right \|_{a_{\text{DG}}\left (K_{j,p-1} \setminus K_{j,p-2} \right )}\left \| \phi _{i,\text{glo}}^{j}  \right \|_{b\left (K_{j,p-1} \setminus K_{j,p-2} \right )}.
    \end{aligned}
\end{equation}
Next, since $\chi _{j}^{p-1,p-2} \equiv 0$ \ in $\Omega \setminus K_{j,p-1}$, we obtain
\begin{equation*}
    \begin{aligned}
    b\left ( \pi\left ( \phi _{i,\text{glo}}^{j}  \right ), \pi\left ( \left ( 1-\chi _{j}^{p-1,p-2} \right )\phi _{i,\text{glo}}^{j}  \right ) \right )&=\left \| \pi\left ( \phi _{i,\text{glo}}^{j}  \right ) \right \|_{b\left ( \Omega \setminus K_{j,p-1} \right )}^{2}\\
&\quad +\int_{K_{j,p-1}\setminus K_{j,p-2}}k_{1}\pi\left ( \phi _{i,\text{glo}}^{j}  \right )\pi\left ( \left ( 1-\chi _{j}^{p-1,p-2} \right )\phi _{i,\text{glo}}^{j}  \right )\text{dx}.
    \end{aligned}
\end{equation*}
Then we note that
\begin{equation}
    \begin{aligned}
&\quad  \left \|  \pi\left ( \phi _{i,\text{glo}}^{j}  \right ) \right \|_{b\left ( \Omega \setminus K_{j,p-1} \right )}^{2}\\
    &=b\left ( \pi\left ( \phi _{i,\text{glo}}^{j}  \right ), \pi\left ( \left ( 1-\chi _{j}^{p-1,p-2} \right )\phi _{i,\text{glo}}^{j}  \right ) \right ) -\int_{K_{j,p-1}\setminus K_{j,p-2}}k_{1}\pi\left ( \phi _{i,\text{glo}}^{j}  \right )\pi\left ( \left ( 1-\chi _{j}^{p-1,p-2} \right )\phi _{i,\text{glo}}^{j}  \right )\text{dx}\\
    &\leq b\left ( \pi\left ( \phi _{i,\text{glo}}^{j}  \right ), \pi\left ( \left ( 1-\chi _{j}^{p-1,p-2} \right )\phi _{i,\text{glo}}^{j}  \right ) \right )+\left \| \phi _{i,\text{glo}}^{j}  \right \|_{b\left ( K_{j,p-1}\setminus  K_{j,p-2}\right )}\left \|\pi\left (  \phi _{i,\text{glo}}^{j}  \right ) \right \|_{b\left ( K_{j,p-1}\setminus  K_{j,p-2}\right )}.
    \end{aligned}
\end{equation}

Finally, we  obtain
\begin{equation}
    \begin{aligned}
    &\quad \left \| \phi _{i,\text{glo}}^{j}  \right \|_{a_{\text{DG}}\left ( \Omega \setminus K_{j,p-1} \right )}^{2}+\left \|  \pi\left ( \phi _{i,\text{glo}}^{j}  \right ) \right \|_{b\left ( \Omega \setminus K_{j,p-1} \right )}^{2}\\
    &\leq \left \| \phi _{i,\text{glo}}^{j}  \right \|_{b\left ( K_{j,p-1}\setminus  K_{j,p-2}\right )}  \left( \left \|\pi\left (  \phi _{i,\text{glo}}^{j}  \right ) \right \|_{b\left ( K_{j,p-1}\setminus  K_{j,p-2}\right )}+\left \| \phi _{i,\text{glo}}^{j}  \right \|_{a_{\text{DG}}\left ( K_{j,p-1}\setminus  K_{j,p-2}\right )} \right)\\
    &\leq \left ( 2+\frac{1}{2}c_{1}^{-\frac{1}{2}} \right ) \left ( \left \|\pi\left (  \phi _{i,\text{glo}}^{j}  \right ) \right \|^{2}_{b\left ( K_{j,p-1}\setminus  K_{j,p-2}\right )}+\left \| \phi _{i,\text{glo}}^{j}  \right \|^{2}_{a_{\text{DG}}\left ( K_{j,p-1}\setminus  K_{j,p-2}\right )}\| \right ).\label{equ:79}
    \end{aligned}
\end{equation}\\
$\mathbf{Part\ 3.}$\ We will show the recursive property of the method in this paper. Firstly we consider $ \left \|\pi\left (  \phi _{i,\text{glo}}^{j}  \right ) \right \|_{b}$ and $\left \| \phi _{i,\text{glo}}^{j}  \right \|_{\text{DG}}$ on $ \Omega\setminus  K_{j,p-2}$ in two regions $\Omega\setminus  K_{j,p-1}$ and $ K_{j,p-1}\setminus  K_{j,p-2}$ and combine (\ref{equ:79}) to get the following results.
\begin{equation}
\begin{aligned}
&\quad \left \|\pi\left (  \phi _{i,\text{glo}}^{j}  \right ) \right \|^{2}_{b\left ( \Omega \setminus  K_{j,p-2}\right )}+\left \| \phi _{i,\text{glo}}^{j}  \right \|^{2}_{a_{\text{DG}}\left ( \Omega\setminus  K_{j,p-2}\right )}\\
&=\left \|\pi\left (  \phi _{i,\text{glo}}^{j}  \right ) \right \|^{2}_{b\left ( \Omega \setminus  K_{j,p-1}\right )}+\left \| \phi _{i,\text{glo}}^{j}  \right \|^{2}_{a_{\text{DG}}\left ( \Omega\setminus  K_{j,p-1}\right )}\\
&\quad +\left \|\pi\left (  \phi _{i,\text{glo}}^{j}  \right ) \right \|^{2}_{b\left ( K_{j,p-1}\setminus  K_{j,p-2}\right )}+\left \| \phi _{i,\text{glo}}^{j}  \right \|^{2}_{a_{\text{DG}}\left ( K_{j,p-1}\setminus  K_{j,p-2}\right )}\\
&\geq\left ( 1+ \left ( 2+\frac{1}{2}c_{1}^{-\frac{1}{2}} \right )^{-1} \right )\left( \left \|\pi\left (  \phi _{i,\text{glo}}^{j}  \right ) \right \|^{2}_{b\left ( \Omega \setminus  K_{j,p}\right )}+\left \| \phi _{i,\text{glo}}^{j}  \right \|^{2}_{a_{\text{DG}}\left ( \Omega\setminus  K_{j,p}\right )}\right).
\end{aligned}
\end{equation}
Using the above inequality recursively, we have
\begin{equation*}
    \begin{aligned}
    &\quad \left \|\pi\left (  \phi _{i,\text{glo}}^{j}  \right ) \right \|_{b\left ( \Omega \setminus  K_{j,p-1}\right )}+\left \| \phi _{i,\text{glo}}^{j}  \right \|_{a_{\text{DG}}\left ( \Omega\setminus  K_{j,p-1}\right )}\\
   & \leq \left ( 1+ \left ( 2+\frac{1}{2}c_{1}^{-\frac{1}{2}} \right )^{-1} \right )^{1-p}\left(\left \|\pi\left (  \phi _{i,\text{glo}}^{j}  \right ) \right \|_{b}+\left \| \phi _{i,\text{glo}}^{j}  \right \|_{\text{DG}}\right).
    \end{aligned}
\end{equation*}
This completes the proof.
\end{proof}

Finally, we prove the convergence.
\begin{theorem}\label{thm:2}
Let $u$ be the solution of (\ref{adg}) and $u_{ms}^{off}$ be the solution of (\ref{eq:ms}).  Then we have
\begin{equation}
    \left \| u-u_{\text{ms}}^{\text{off}} \right \|_{\text{DG}}\leq c_{3}c_{1}^{-\frac{1}{2}}\left \| k_{1}^{-\frac{1}{2}}f \right \|+c_{3}\left ( p+1 \right )^{\frac{d}{2}}\sqrt{M_{2}\left(1+M\right)}\left \| u_{\text{gl}} \right \|_b,
\end{equation}
where $u_{\text{gl}}$ is the multiscale solution using correspongding global basis in (\ref{equ:global}). \\
Moreover, if $p=O\left(\log \left(\frac{\lambda _{\max}\left ( C\left ( x \right ) \right )}{H}\right)\right)$ and $\chi_{i}$ are bilinear partition of unity, then there exists a constant $c_4$ such that 
\begin{equation}
\left \| u-u_{\text{ms}}^{\text{off}} \right \|_{\text{DG}} \leq c_{4}c_{1}^{-\frac{1}{2}}H\left \| k_{2}^{-\frac{1}{2}}f \right \| .
\end{equation}
\end{theorem}

\begin{proof}
Define $u_{\text{gl}}=\sum_{j=1}^{N}\sum_{i=1}^{G_{j}}d_{ij}\phi _{i,\text{glo}}^{j}$ and $v=\sum_{j=1}^{N}\sum_{i=1}^{G_{j}}d_{ij} \phi _{i,\text{ms}}^{j}$.
So, by the triangular inequality, Lemma \autoref{lemma4.4} and Lemma \autoref{lemma4.5}, we can get
\begin{equation}
    \begin{aligned}
\left \| u-u_{\text{ms}}^{\text{off}} \right \|_{\text{DG}}^{2}
&\leq \left \| u-u_{\text{gl}} \right \|_{\text{DG}}^{2}+ \left \| u_{\text{gl}} -v\right \|_{\text{DG}}^{2}\\
&\leq \left\|u-u_{\text{gl}}\right\|_{\text{DG}}^{2}+\left \| \sum_{j=1}^{N}\sum_{i=1}^{l_{i}}d_{ij}\left ( \phi _{i,\text{ms}}^{j} - \phi _{i,\text{glo}}^{j}\right ) \right \|_{\text{DG}}^{2}\\
&\leq \left\|u-u_{\text{gl}}\right\|_{\text{DG}}^2+c_{3}\left ( p+1 \right )^{d}\sum_{j=1}^{N}\left \|\sum_{i=1}^{l_{i}}d_{ij}\left ( \phi _{i,\text{ms}}^{j} - \phi _{i,\text{glo}}^{j}\right ) \right \|_{\text{DG}}^{2}\\
&\leq \left\|u-u_{\text{gl}}\right\|_{\text{DG}}^2+c_{3}\left ( p+1 \right )^{d}M_{2}\sum_{j=1}^{N}\sum_{i=1}^{l_{i}}\left \|d_{ij}\psi _{i}^{j} \right \|_{b}^{2}\\
&\leq\left\|u-u_{\text{gl}}\right\|_{\text{DG}}^2+ c_{3}\left ( p+1 \right )^{d}M_{2}\sum_{j=1}^{N}\sum_{i=1}^{l_{i}}\left(d_{ij}\right)^{2}.
    \end{aligned}
\end{equation}
By the definition of the global multiscale basis functions and the projection $\pi$, we have
$$\pi\left( u_{\text{gl}}\right)=\sum_{j=1}^{N}\sum_{i=1}^{l_{i}}d_{ij}\pi \left(\phi _{i,\text{glo}}^{j}\right).$$
Since $\sum_{j=1}^{N}\sum_{i=1}^{l_{i}}d_{ij}a_{\text{DG}}\left ( \phi _{i,\text{ms}}^{j},\psi _{p}^{l} \right )=0, $ we obtain
\begin{equation*}
    \begin{aligned}
    b\left ( \pi \left(u_{\text{gl}}\right),\psi _{p}^{l} \right )
    &=\sum_{j=1}^{N}\sum_{i=1}^{l_{i}}d_{ij}b\left ( \pi\left( \phi _{i,\text{ms}}^{j}\right),\psi _{p}^{l} \right )\\
    &=\sum_{j=1}^{N}\sum_{i=1}^{l_{i}}d_{ij}\left ( b\left ( \pi \left ( \phi _{i,\text{ms}}^{j}\right ),\pi \left ( \psi _{p}^{l} \right )  \right )+a_{\text{DG}}\left ( \phi _{i,\text{ms}}^{j},\psi _{p}^{l} \right ) \right ).
    \end{aligned}
\end{equation*}
To simplify our inequalities,  let $b\left ( \pi\left( u_{\text{gl}}\right),\psi _{p}^{l} \right )$ compose a matrix $B$, $\left ( b\left ( \pi\left ( \phi _{i,\text{ms}}^{j} \right )\pi\left ( \psi _{p}^{l} \right ) \right )+a_{\text{DG}}\left ( \phi _{i, \text{ms}}^{j}, \psi _{p}^{l} \right ) \right )$ compose a matrix $A$, and $\left(d_{ij}\right)$ compose a matrix $D$. Then we have
$$\left \| D \right \|_{2}\left \| A \right \|_{2}\leq \left \| B \right \|_{2},$$ that is 
$$\left \| D \right \|_{2}\leq \left \| A^{-1} \right \|_{2}\left \| B \right \|_{2}.$$
Then by the definition of the space $V_h$ and based on the method we used in this paper, for $\psi  =\sum_{j=1}^{N}\sum_{i=1}^{l_{i}}d_{ij}\psi _{i}^{j}\in W_{\text{aux}},$ there is $\phi =\sum_{j=1}^{N}\sum_{i=1}^{l_{i}}d_{ij}\phi  _{i,\text{glo}}^{j}$ such that 
\begin{equation}
a_{\text{DG}}\left ( \phi ,v \right )+b\left ( \pi\phi ,\pi v \right )=b\left ( \psi ,\pi v \right ),\ \forall\ v\in V_h.\label{equ:85}
\end{equation}
Using the given $\psi\in W_{\text{aux}}$, there is $y\in V_h$ such that 
\begin{equation*}
\pi y=\psi, \left \| y \right \|_{\text{DG}}^{2}\leq M\left \| \phi  \right \|_{b}^{2}. 
\end{equation*}
Let $v=y$ in (\ref{equ:85}), we can similarly note that
$$a_{\text{DG}}\left ( \phi ,y \right )+b\left ( \pi\phi ,\pi y \right )=b\left ( \psi ,\pi y \right ).$$
Notice that $b\left ( \psi ,\pi y \right )=b\left ( \psi, \psi \right )=\left \| D \right \|_{2}^{2},$
hence we obtain
\begin{equation*}
    \begin{aligned}
    \left \| D \right \|_{2}^{2}&=a_{\text{DG}}\left ( \phi ,y \right )+b\left ( \pi \phi ,\psi  \right )\\
    &\leq \left \| \phi  \right \|_{\text{DG}}\left \| y \right \|_{\text{DG}}+\left \| \pi \phi  \right \|_{\text{DG}}\left \| \psi  \right \|_{\text{DG}}\\
    &\leq \sqrt{1+M}\left \| \psi  \right \|_{b}\left ( \left \| \phi  \right \|_{\text{DG}}^{2}+\left \| \pi \phi  \right \|_{b}^{2} \right )^{\frac{1}{2}}\\
    &\leq \left ( 1+M \right )\left \| B \right \|_{2}^{2}\\
    & \leq \left ( 1+M \right )\left \| u_{\text{gl}} \right \|_{b}^{2}.
    \end{aligned}
\end{equation*}
Then we have
$$ \left \| u_{\text{gl}}-v \right \|_{\text{DG}}^{2}\leq c_{3}\left ( p+1 \right )^{d}M_{2}\left ( 1+M \right )\left \| u_{\text{gl}} \right \|_{b}^{2}.$$
 The rest proof is similarly with \autoref{thm:thm1}. This completes the proof.
\end{proof}

We remark that \autoref{thm:2} gives a better constraint than  \autoref{thm:thm1}, because the convergence rate in \autoref{thm:2} is independent of the constant $M$. This is supported by the numerical data shown below.

\section{Numerical results}\label{section5}
  
In this section, we present numerical simulation results to show the performance of the proposed method. We consider a complicated high-contrast elastic model whose Young's modulus  $E\left(x\right)$ is shown in \autoref{Young1}, besides, we let 
$$\lambda\left(x\right)=\frac{v}{\left(1+v\right)\left(1-2 v\right)} E\left(x\right),\ \mu\left(x\right)=\frac{1}{2\left(1+v\right)} E\left(x\right). $$
The Poisson's ratio $v$ is 0.25. 
In all numerical tests, we use constant force and homogeneous Dirichlet boundary conditions. 
In our simulations, we divide the domain $D=[0,1]\times [0,1]$ into $15\times 15$ coarse grid blocks, together with in each coarse block we use $15\times 15$ fine-scale square blocks, which results in $225\times 225$ fine grid blocks. 
Then the degrees of freedom used to obtain the reference solution is $115,200$.
In order to evaluate the accuracy of the multiscale solution, following errors  are adopted:
$$e_{L^{2}}=\frac{\left\|(\lambda+2 \mu)\left(u_{\text{ms}}-u_{h}\right)\right\|_{L^{2}(D)}}{\left\|(\lambda+2 \mu) u_{h}\right\|_{L^{2}(D)}}, \quad e_{H^{1}}=\sqrt{\frac{a\left(u_{\text{ms}}-u_{h}, u_{\text{ms}}-u_{h}\right)}{a\left(u_{h}, u_{h}\right)}}.$$

For the convenience of presenting the simulation results,  we provide  some notations in  \autoref{ta:symbol}.
\begin{table}[H]
	
\setlength{\abovecaptionskip}{0pt}
\setlength{\belowcaptionskip}{10pt}
\centering
\topcaption{\label{ta:symbol}Simplified description of symbols.}
\label{ta:symbol}
\begin{tabular}{l|llllll}
\cline{1-2}
Parameters                                        & Symbols &  &  &  &  &  \\ \cline{1-2}
Number of oversampling coarse element                & $Noc$      &  &  &  &  &  \\
Number of basis functions in each coarse element & $Nbf $    &  &  &  &  &  \\
Size of each coarse element                & $H $      &  &  &  &  &  \\
Contrast value                                    & $E_{1}$      &  &  &  &  &  \\ \cline{1-2}
\end{tabular}
\end{table}

In the standard test, the number of basis function and oversampled element are 6 and 5, respectively, the contrast of the model is $10^4$. The FEM solution and the relaxed CEM-GMsFEM solution with standard parameter settings are shown in \autoref{refsol} and \autoref{relaxedsol}. We can see there is obvious distortion in the reference solution due to the strong heterogeneity of the media, however, the 
coarse-grid solution still succeed in capturing those small details. Actually, it is hard to find any differences between the multiscale solution and the reference solution.
\begin{figure}[H] 
\centering 
\includegraphics[width=0.7\textwidth]{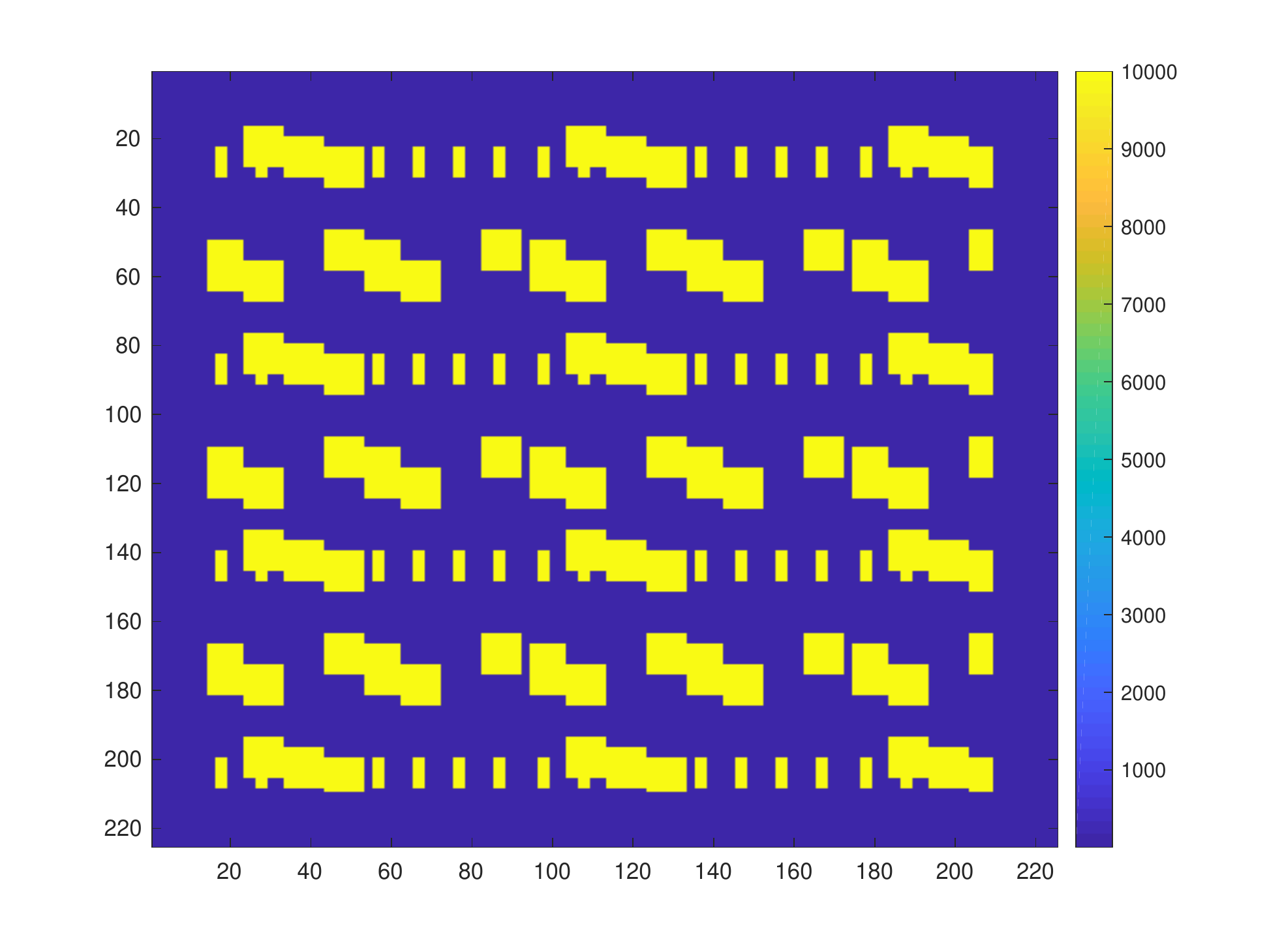} 
\caption{Young's Modulus of the test model.}
\label{Young1} 
\end{figure}

\begin{figure}[H] 
\centering{\includegraphics[width=16cm]{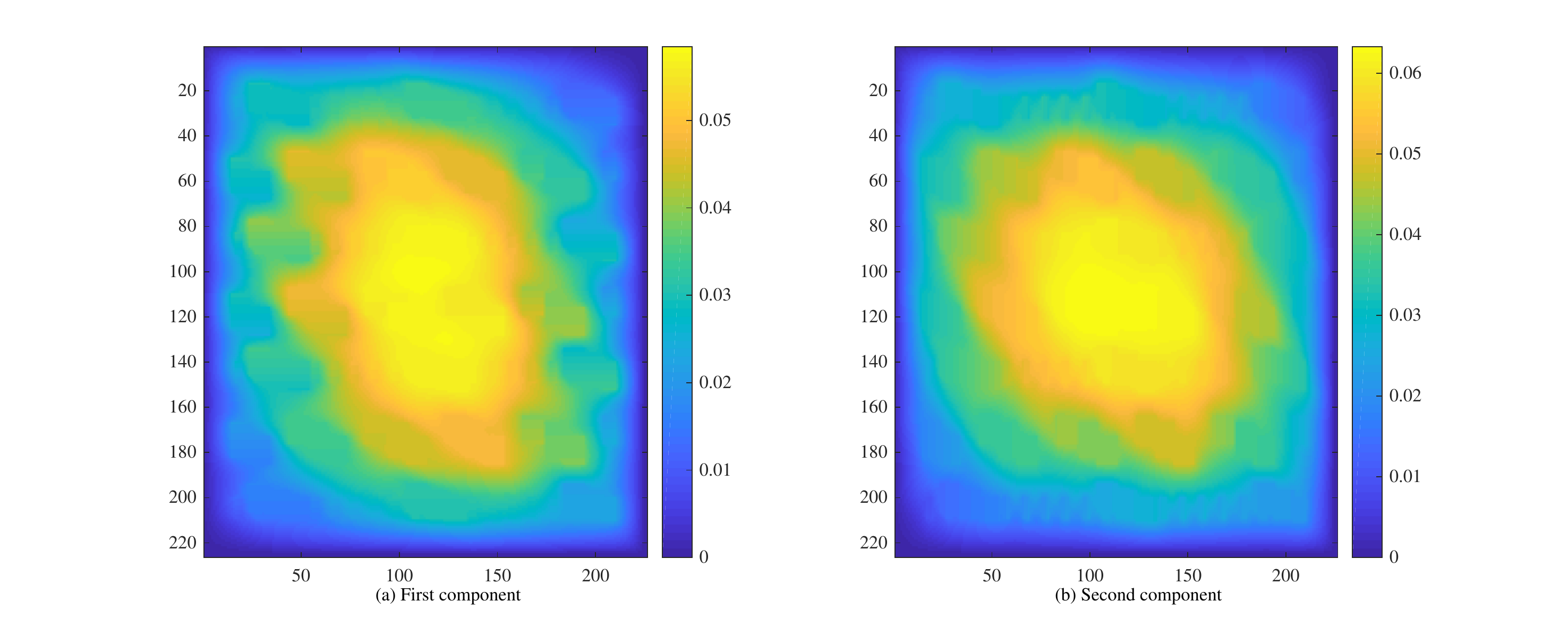} }
\caption{Reference solution.}
\label{refsol} 
\end{figure}

\begin{figure}[H] 
\centering{\includegraphics[width=16cm]{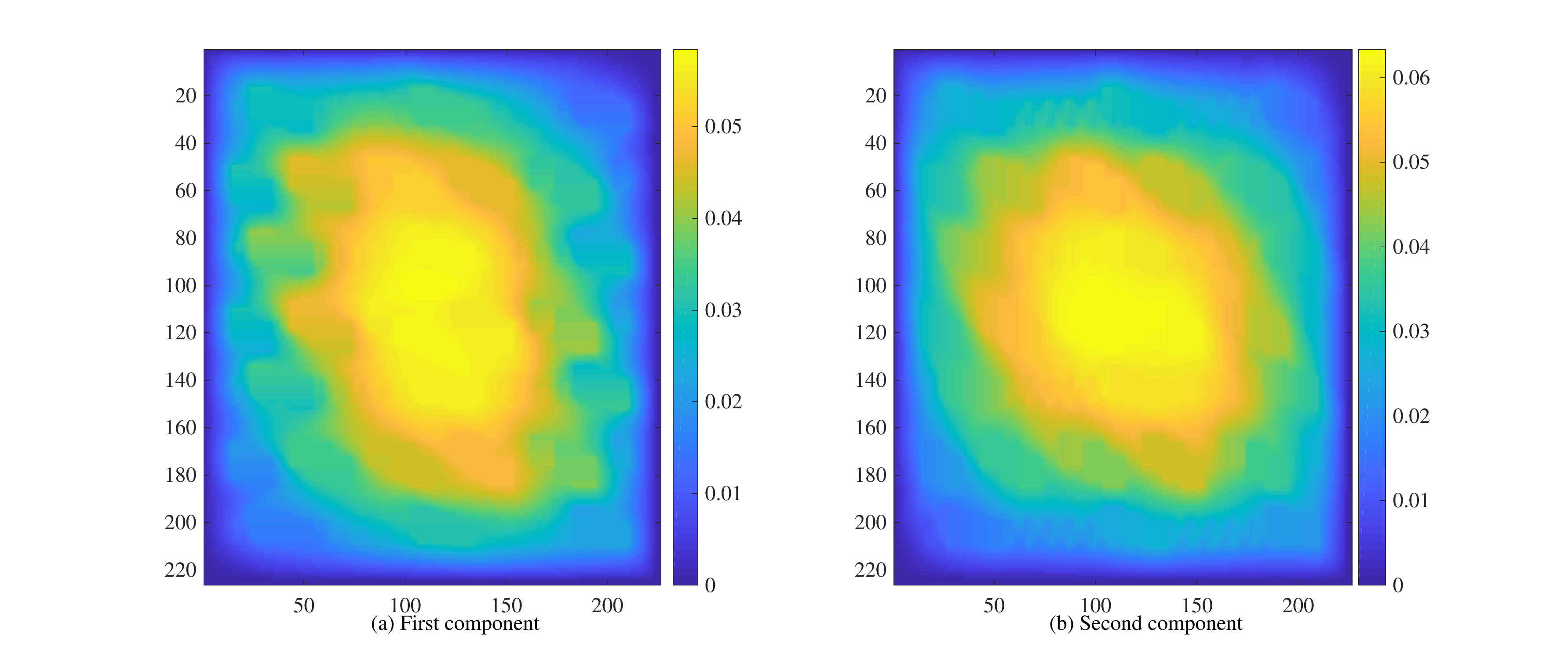} }
\caption{Relaxed CEM-GMsFEM solution with standard parameter setting.}
\label{relaxedsol} 
\end{figure}

To study the effects of different number of basis functions on the accuracy of the GMsFEM solution, we display the errors with respect to the number of basis functions in \autoref{table-Nbf} and \autoref{fig-Nbf}, we note that in these tests, we let $Noc=5,$  $H=1/60$.   
It is clear that more number of basis functions will definitely yield more accurate GMsFEM solutions. For example, we can observe that if only 1 basis was used in the Relaxed CEM-GMsFEM, the $H^1$ error is about 26\%, which decays to only 2\% 
if there are 8 bases. We note the $L^2$ error the Relaxed version is 0.217\% if $Nbf=8$, which is also much smaller than the case of $Nbf=1$.

\begin{figure}[H] 
\centering 
\includegraphics[width=0.7\textwidth]{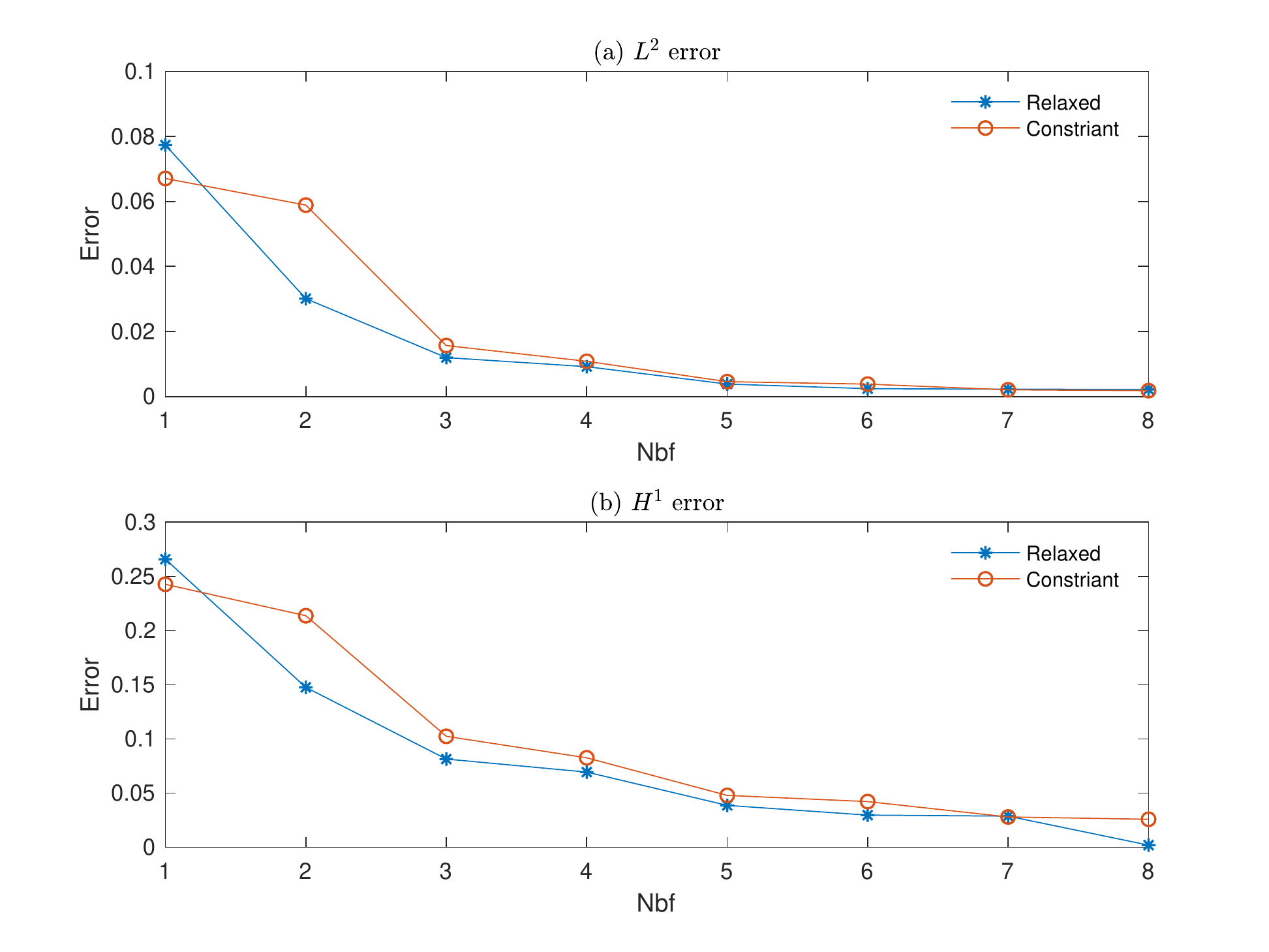} 
\caption{Comparison of the Relaxed and  Constrained CEM-GMsFEM with different $Nbf,$  $H=1/60,$  $Noc=5$.}
\label{fig-Nbf} 
\end{figure}

\begin{table}[H]
\centering
\topcaption{Comparison of the Relaxed and  Constrained CEM-GMsFEM with different $Nbf,$ $H=1/60,$ $Noc=5$.}
\begin{tabular}{cllccllcclll}
\cline{1-10}
\multicolumn{2}{c}{$Nbf$} &  & \multicolumn{3}{c}{$e_{L^{2}}$}             &  & \multicolumn{3}{c}{$e_{H^{1}}$}             &  &  \\ \cline{4-6} \cline{8-10}
\multicolumn{2}{c}{}                                 &  & Relaxed                     &                      &  Constraint  &  & Relaxed                     &                      &  Constraint  &  &  \\ \cline{1-10}
\multicolumn{2}{c}{1}                                &  & 7.734\%                     &                      & 6.706\% &  & 26.586\%                    &                      &  24.275\% &  &  \\
\multicolumn{2}{c}{2}                                &  & 3.009\%                     &                      & 5.886\% &  & 14.764\%                    &                      & 21.373\% &  &  \\
\multicolumn{2}{c}{3}                                &  & 1.196\%                     &                      & 1.567\% &  &  \multicolumn{1}{l}{8.147\%}                      &                      &  10.247\% &  &  \\
\multicolumn{2}{c}{4}                                &  & 0.917\%                     &                      & 1.082\% &  &  \multicolumn{1}{l}{6.945\%}                      &                      & 8.261\% &  &  \\
\multicolumn{2}{c}{5}                                &  & \multicolumn{1}{l}{0.382\%} & \multicolumn{1}{l}{} &  0.456\% &  & \multicolumn{1}{l}{3.872\%} & \multicolumn{1}{l}{} & 4.792\% &  &  \\
\multicolumn{2}{c}{6}                                &  & \multicolumn{1}{l}{0.240\%} & \multicolumn{1}{l}{} &  0.412\% &  & \multicolumn{1}{l}{2.973\%} & \multicolumn{1}{l}{} & 3.282\% &  &  \\
\multicolumn{2}{c}{7}                                &  & \multicolumn{1}{l}{0.224\%} & \multicolumn{1}{l}{} & 0.205\% &  & \multicolumn{1}{l}{2.880\%} & \multicolumn{1}{l}{} & 2.804\% &  &  \\
\multicolumn{2}{c}{8}                                &  & \multicolumn{1}{l}{0.217\%} & \multicolumn{1}{l}{} & 0.179\% &  & \multicolumn{1}{l}{2.050\%}  & \multicolumn{1}{l}{} &  2.589\% &  &  \\ \cline{1-10}
\end{tabular}\label{table-Nbf}
\end{table}

It is also important  to investigate the influence of  oversampling size. To this end, we compare the error as a function of  $Noc$  in \autoref{table-Noc} and \autoref{fig-Noc} with fixed $Nbf=4$ and $\ H=1/15$. 
It can be found there is sharp decay of the error once $Noc$ reaches 3 for the 
Relaxed scheme and 5 for the Constraint scheme. Once $Noc$ is larger than a 
threshold, accuracy improvements are no longer significant.

\begin{figure}[H] 
\centering 
\includegraphics[width=0.7\textwidth]{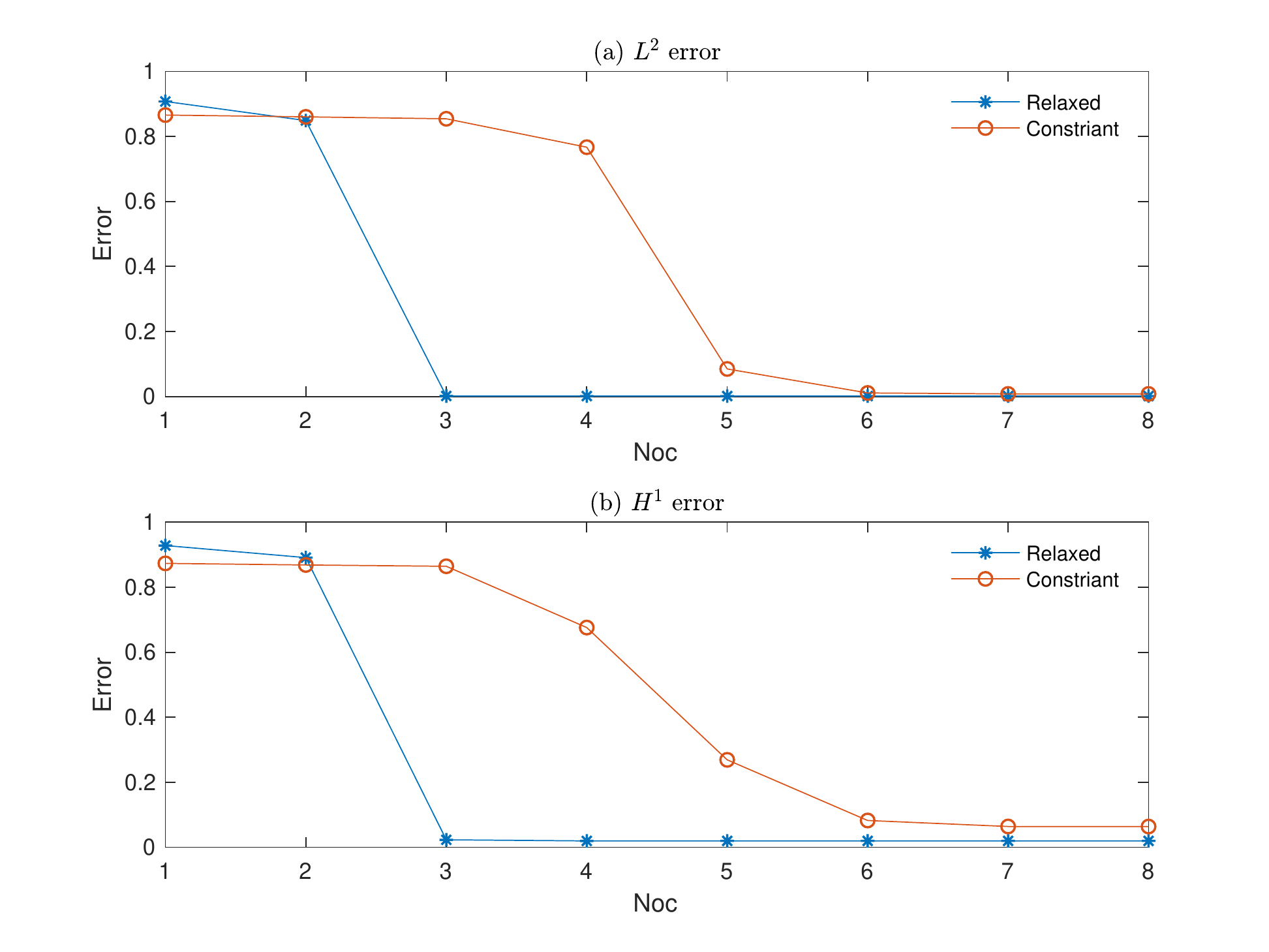} 
\caption{Comparison of the Relaxed and Constrained CEM-GMsFEM with different $Noc,$ $H=1/15,$ $Nbf=4$.}
\label{fig-Noc} 
\end{figure}

\begin{table}[H]
\centering
\topcaption{Comparison of the Relaxed and Constrained CEM-GMsFEM with different $Noc,$ $H=1/15,$ $Nbf=4$.}
\begin{tabular}{cllccllcclll}
\cline{1-10}
\multicolumn{2}{c}{$Noc$} &  & \multicolumn{3}{c}{$e_{L^{2}}$}             &  & \multicolumn{3}{c}{$e_{H^{1}}$}             &  &  \\ \cline{4-6} \cline{8-10}
\multicolumn{2}{c}{}                     &  & Relaxed                     &                      & Constraint  &  & Relaxed                     &                      & Constraint  &  &  \\ \cline{1-10}
\multicolumn{2}{c}{1}                    &  & 90.743\%                    &                      &  86.559\% &  & 92.829\%                    &                      & 87.341\% &  &  \\
\multicolumn{2}{c}{2}                    &  & 84.869\%                    &                      & 86.006\%  &  & 89.114\%                    &                      & 86.850\% &  &  \\
\multicolumn{2}{c}{3}                    &  & \multicolumn{1}{l}{0.151\%}                    &                      & 85.412\% &  &  \multicolumn{1}{l}{2.322\%}                     &                      & 86.449\% &  &  \\
\multicolumn{2}{c}{4}                    &  & \multicolumn{1}{l}{0.127\%}                    &                      &  76.676\% &  &  \multicolumn{1}{l}{1.974\%}                     &                      & 67.605\% &  &  \\
\multicolumn{2}{c}{5}                    &  & \multicolumn{1}{l}{0.127\%} & \multicolumn{1}{l}{} &   8.470\% &  & \multicolumn{1}{l}{1.978\%} & \multicolumn{1}{l}{} &  26.915\% &  &  \\
\multicolumn{2}{c}{6}                    &  & \multicolumn{1}{l}{0.127\%} & \multicolumn{1}{l}{} &   1.082\% &  & \multicolumn{1}{l}{1.978\%} & \multicolumn{1}{l}{} & 8.261\% &  &  \\
\multicolumn{2}{c}{7}                    &  & \multicolumn{1}{l}{0.127\%} & \multicolumn{1}{l}{} &   0.787\% &  & \multicolumn{1}{l}{1.978\%} & \multicolumn{1}{l}{} & 6.417\% &  &  \\
\multicolumn{2}{c}{8}                    &  & \multicolumn{1}{l}{0.127\%} & \multicolumn{1}{l}{} &   0.785\% &  & \multicolumn{1}{l}{1.978\%} & \multicolumn{1}{l}{} & 6.396\% &  &  \\ \cline{1-10}
\end{tabular}\label{table-Noc}
\end{table}

We also examine the robustness of the proposed method against the contrast of the model. We fix $H$ and $Nbf$ and vary the ratio of the maximum value of $E(x)$ over its minimum value and $Noc$. The simulation results are reported in \autoref{table-contrast}. We find that larger $Noc$ leads to more robust coarse-grid simulations, the Relaxed formulation is more robust than the Constraint version. 

\begin{table}[H]
\centering 
\topcaption{Comparison weighted errors of different $Noc$ and contrast values with $H=1/15$, $Nbf=6$. }
\begin{tabular}{cclllllll}
\cline{1-7}
$E_{1}$                       & Noc &  &  & \multicolumn{3}{c}{$e_{H^{1}}$}             &  &  \\ \cline{5-7}
                         &     &  &  & \multicolumn{1}{c}{Relaxed} & \multicolumn{1}{c}{} & Constraint &  &  \\ \cline{1-7}
                         & 5   &  &  & 3.872\%                     & \multicolumn{1}{c}{} &        14.792\% &  &  \\
$10^{2}    $                  & 6   &  &  & 2.973\%                     & \multicolumn{1}{c}{} &   4.223\% &  &  \\
                         & 7   &  &  & 2.880\%                     & \multicolumn{1}{c}{} &        2.804\% &  &  \\ \cline{1-7}
                         & 5   &  &  & 2.927\%                     & \multicolumn{1}{c}{} &        4.122\% &  &  \\
$10^{4}  $                   & 6   &  &  & 2.928\%                     &                      &  2.987\%  &  &  \\
                         & 7   &  &  & 2.928\%                     &                      &  2.967\%  &  &  \\ \cline{1-7}
                         & 5   &  &  & 2.783\%                     &                      &   89.107\%  &  &  \\
$10^{6}  $                    & 6   &  &  & 2.783\%                     &                      &  50.699\%   &  &  \\
\multicolumn{1}{l}{}     & 7   &  &  & 0.211\%                     &                      &  7.470\%  &  &  \\ \cline{1-7}
\multicolumn{1}{l}{}     & 5   &  &  & 3.295\%                     &                      &   95.676\%  &  &  \\
\multicolumn{1}{l}{$10^{8}$}  & 6   &  &  & 3.197\%                     &                      &  85.561\%   &  &  \\
\multicolumn{1}{l}{}     & 7   &  &  & 3.155\%                     &                      &   47.624\%  &  &  \\ \cline{1-7}
\end{tabular}\label{table-contrast}
\end{table}

Finally, we study the convergence behavior of the CEM-GMsFEM solution with respect to the coarse-grid size $H$. We set the number of oversampling layers be $Noc\approx 4\lfloor\log \left(H\right) / \log \left(1/7.5\right)\rfloor$ and $Nbf=4$ to form the basis spaces. The  convergence results are presented in \autoref{table-L}. In \autoref{loghfinal}, we can see  the $H^1$ error of  CEM-GMsFEM solutions almost linearly converge against the coarse-mesh size $H$ as predicted by the analysis.
\begin{figure}[H] 
\centering 
\includegraphics[width=0.7\textwidth]{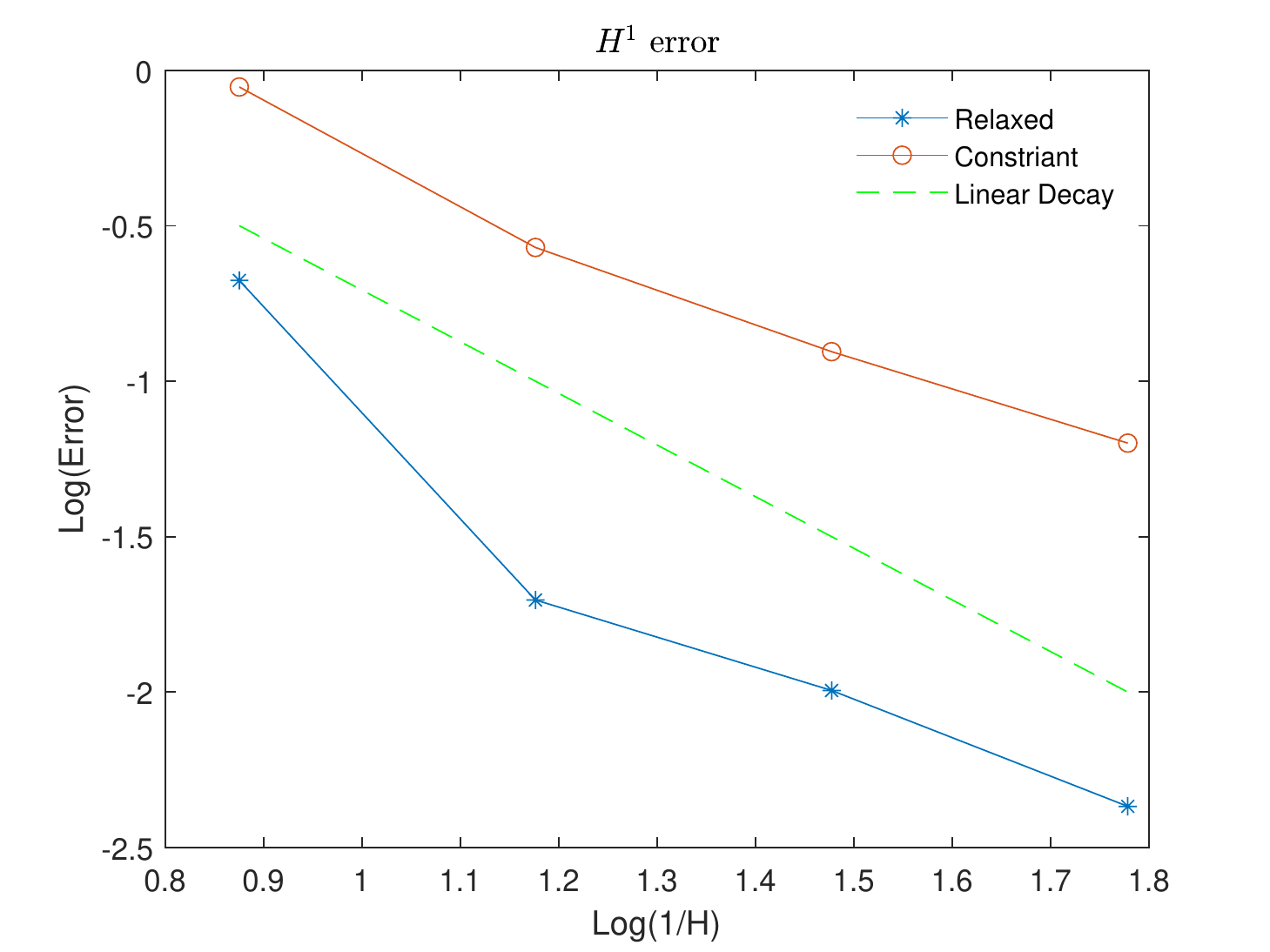} 
\caption{Comparison of the Relaxed and Constrained CEM-GMsFEM with different $Noc$ and $H,$  $Nbf=4$.}
\label{loghfinal} 
\end{figure}

\begin{table}[H]
\centering
\topcaption{Comparison of the Relaxed and Constrained CEM-GMsFEM with different $Noc$ and $H,$  $Nbf=4$.}
\begin{tabular}{cclccllcclll}
\cline{1-10}
\multicolumn{2}{c}{Dimension}      &  & \multicolumn{3}{c}{$e_{L^{2}}$} &   & 
\multicolumn{3}{c}{$e_{H^{1}}$} &  &  \\ \cline{1-2} \cline{4-6} \cline{8-10}
$Noc$ & $H$ &  & Relaxed           &           &  Constraint           &  & Relaxed           &           &  Constraint          &  &  \\ \cline{1-10}
4               & 1/7.5            &  & 4.961\%              &&  17.233\% &  & 21.008\%    && 88.341\% &  &  \\
5               & 1/15             &  & 0.127\%             && 8.470\%  &  & 1.978\%    && 26.915\% &  &  \\
6               & 1/30             &  & 0.071\%               && 3.284\% &  & 1.012\%        && 12.441\% &  &  \\
7               & 1/60             &  & 0.042\%               &&  1.001\% &  & 0.429\%       && 6.320\% &  &  \\          \cline{1-10}
\end{tabular}\label{table-L}
\end{table}

In summary, the relaxed CEM-GMsFEM is more accurate and more robust than the 	Constraint CEM-GMsFEM. Besides, more basis functions and larger local problem size lead to more accurate coarse-grid solutions.



\section{Conclusion}\label{section6}
In this paper, we introduce the CEM-GMsFEM with the discontinuous Galerkin form, which is a local multiscale model reduction method in the discontinuous Galerkin framework and can be used to solve the elasticity problem of high contrast. Firstly, we use the main mode from the local spectrum problem to construct a multiscale basis function to locally represent the solution space in each coarse block. Then, the multiscale basis function in the coarse oversampling region is solved by constraining the energy minimization problem. After proving the convergence of the CEM-GMsFEM and the Relaxed CEM-GMsFEM in the DG form, we also give the numerical results of a simple elasticity equation to further illustrate the method's accuracy.
\section*{Acknowledgments}

The research of Eric Chung is partially supported by the Hong Kong RGC General Research Fund (Project numbers 14304719 and 14302620) and CUHK Faculty of Science Direct Grant 2021-22.

\small

\bibliography{p1}

\end{document}